\documentclass[12pt,a4paper,reqno]{amsart}

\usepackage{amsmath}
\usepackage{amsfonts}
\usepackage{amssymb}
\usepackage[dvipsnames]{xcolor}
\usepackage{mathrsfs}
\usepackage{aliascnt} 
\usepackage{hyperref}
\usepackage{cleveref}
\usepackage{enumitem}

\usepackage{cite}

\usepackage[foot]{amsaddr}

\newtheorem{theorem}{Theorem}[section]
\crefname{theorem}{Theorem}{Theorems}

\newaliascnt{lemma}{theorem}
\newtheorem{lemma}[lemma]{Lemma}
\aliascntresetthe{lemma}
\crefname{lemma}{Lemma}{Lemmas}

\newaliascnt{definition}{theorem}
\newtheorem{definition}[definition]{Definition}
\aliascntresetthe{definition}
\crefname{definition}{Definition}{Definitions}

\newaliascnt{prop}{theorem}
\newtheorem{prop}[prop]{Proposition}
\aliascntresetthe{prop}
\crefname{prop}{Proposition}{Propositions}

\newaliascnt{corollary}{theorem}

\aliascntresetthe{corollary}
\crefname{corollary}{Corollary}{Corollaries}

\usepackage{aliascnt}
\newaliascnt{remark}{theorem}
\newtheorem{remark}[remark]{Remark}
\aliascntresetthe{remark}
\crefname{remark}{Remark}{Remarks}

\newaliascnt{hypotheses}{theorem}

\aliascntresetthe{hypotheses}
\crefname{hypotheses}{Hypotheses}{Hypotheses}

\numberwithin{equation}{section}

\begin{document}
	\begin{center}
		\Large{\textbf{On the well-posedness and long-time dynamics of a discrete coagulation--annihilation system}}
	\end{center}
    	\medskip

\centerline{${\text{Ikjot~kaur$^{1}$}}$, {${\text{Saroj ~ Si$^{2}$}}$} and  ${\text{Ankik Kumar Giri$^{3,*}$}}$}
\let\thefootnote\relax
\footnotetext{$^{*}$Corresponding author. Tel +91-1332-284818 (O);  Fax: +91-1332-273560  \newline{\it{${}$ \hspace{.3cm} Email address: }}ankik.giri@ma.iitr.ac.in}
	\medskip
	{\footnotesize
		\centerline{$^{1,3}$Department of Mathematics, Indian Institute of Technology Roorkee,}
		
	\centerline{Roorkee-247667, Uttarakhand, India}
    \centerline{$^{2}$Department of Mathematics and Applied Mathematics University of Pretoria,}
		
	\centerline{Pretoria-0028, Gauteng, South Africa}
	}
	\begin{quote}
	{\small {\em \bf Abstract.} The present article investigates the well--posedness of a discrete coagulation--annihilation model originally introduced by E. Ben-Naim and P. Krapivsky in Physical Review E (1995). We first establish the existence of solutions for a broad class of coagulation and annihilation rates, including rates that may be unbounded. These results are obtained by combining uniform estimates with Helly's selection theorem and the Rellich--Kondrachov compact embedding theorem. Under appropriate structural assumptions on these rates, we further prove uniqueness and continuous dependence of solutions on the initial data, thereby establishing well--posedness of the model. Finally, for a specific class of coagulation and annihilation rates, we additionally study the differentiability of solutions and characterize their long--time asymptotic behaviour. These results provide a rigorous mathematical framework for the analysis of the model and contribute to a deeper understanding of its well--posedness, regularity, and long-time dynamics.} 
\end{quote}
		\vspace{.3cm}
	\section{Introduction}
  \noindent The coagulation process refers to the phenomenon in which clusters of the same species merge to form larger clusters or aggregates. Such processes arise naturally in numerous applications, including aerosol dynamics~\cite{friedlander2000smoke}, polymerization~\cite{ziff1980kinetics}, molecular biology~\cite{Banasiak2019}, and astrophysics~\cite{lissauer1993growth,safronov1972evolution}. However, in many physical and chemical systems, particle interactions are not restricted to identical species. Instead, distinct species may interact, leading to mechanisms that alter the concentrations of both populations. For example, in a two-species system consisting of monomers $A$ and $B$, the aggregation of identical species leads to the formation of open linear polymers, whereas interactions between dissimilar clusters result in the formation of ring polymers. Unlike linear polymers, which remain reactive and continue to participate in subsequent coagulation events, ring polymers are chemically inert and are effectively removed from the reaction process~\cite{ben1995kinetics}. Motivated by such phenomena, the present work explores this class of models. More precisely, let $A_z$ and $B_z$ denote particles of types~$A$ and~$B$, respectively, with size $z$. The fundamental reactions governing the evolution of the system are given by 
	 \begin{align*}
	 A_x + A_y \;\longrightarrow\; A_{x+y},\;B_x + B_y\;\longrightarrow\; B_{x+y}\;\text{(coagulation)},\;\text{and}\;
	 A_x + B_y \;\longrightarrow\; \emptyset 
	\; \text{(complete annihilation).}
	 \end{align*}
The following system of equations describes how the concentrations of clusters of size \(i\), denoted by \(a_i(t)\) for clusters of type \(A\) and \(b_i(t)\) for clusters of type \(B\), evolve over time $t$. For \(t>0\), these concentrations satisfy
\begin{align}
	\frac{da_i}{dt} &= \frac{1}{2} \sum_{j=1}^{i-1} K^{a}_{j,i-j}\,a_{j} a_{i-j}
	- \sum_{j=1}^{\infty} K^{a}_{i,j}\, a_{i} a_{j}
	- \sum_{j=1}^{\infty} J_{i,j}\,a_i b_{j}, \quad i \in \mathbb{N},\label{eqn 1.1}\\
	\frac{db_i}{dt} &= \frac{1}{2} \sum_{j=1}^{i-1} K^{b}_{j,i-j}\,b_{j} b_{i-j}
	- \sum_{j=1}^{\infty} K^{b}_{i,j}\, b_{i} b_{j}
	- \sum_{j=1}^{\infty} J_{i,j}\,b_i a_{j}, \quad i \in \mathbb{N},\label{eqn 1.2}\\
	a_{i}(0) &= a^{0}_{i} \geq 0, \qquad
	b_{i}(0) = b^{0}_{i} \geq 0, \quad i \in \mathbb{N}.\label{eqn 1.3}
\end{align}
\noindent For $i=1$, the leading summation terms in  \eqref{eqn 1.1} and \eqref{eqn 1.2} vanish identically. Here, $\mathbb{N}$ denotes the set of positive integers. Furthermore, $K_{i,j}^{a}$ and $K_{i,j}^{b}$ denote the coagulation rates for clusters of types~$A$ and~$B$, respectively; that is, they represent the rates at which clusters of sizes $i$ and $j$ of the corresponding type merge into an aggregate of size $i+j$. On the other hand, $J_{i,j}$ denotes the annihilation rates for clusters of types~$A$ and~$B$, representing the rates at which clusters of sizes $i$ and $j$ collide and annihilate each other.
 We assume that $K_{i,j}^{a}$, $K_{i,j}^{b}$, and $J_{i,j}$ are real-valued and satisfy
\begin{align}\label{symmetry conditions}
0 \leq K^{a}_{i,j} = K^{b}_{i,j} = K_{i,j} = K_{j,i} 
\quad \text{and} \quad 
0 \leq J_{i,j} = J_{j,i},\;\text{for all}\;i,j \geq 1.
\end{align}
The leading summation terms in \eqref{eqn 1.1} and \eqref{eqn 1.2} represent the formation of clusters of size~$i$ of types~$A$ and~$B$, respectively, through the interaction of smaller clusters of the same type. The second terms represent the depletion of clusters of size~$i$ of types~$A$ and~$B$, respectively, due to coagulation with clusters of the same type and arbitrary size. The third terms represent the loss of clusters of size~$i$ of types~$A$ and~$B$, respectively, due to collisions with clusters of the opposite type and arbitrary size. It should be noted that the model considered above is motivated by the work of E. Ben-Naim and P. Krapivsky~\cite{ben1995kinetics}. In \cite{ben1995kinetics}, E. Ben-Naim et al. obtained an exact solution to \eqref{eqn 1.1}--\eqref{eqn 1.3}, specifically in the case
\begin{align*}
	K_{i,j}^{a} = K_{i,j}^{b} = J_{i,j} = 2,\quad 
	a_{i}(0) = \delta_{i1}, \quad b_{i}(0) = \lambda \delta_{i1},
\end{align*}
for all $i,j \geq 1$ and some parameter $\lambda \geq 0$, where $\delta_{i1}$ denotes the Kronecker delta defined by
\[
\delta_{i1} =
\begin{cases}
	1, & i = 1, \\
	0, & i \neq 1.
\end{cases}
\]
They also investigated whether the solutions exhibit universal self-similar behaviour.
In~\cite{laurenccot2010nonuniversal}, Ph. Lauren{\c{c}}ot and H. Van Roessel studied the continuous coagulation-annihilation system~\eqref{eqn 1.1}--\eqref{eqn 1.3} with constant coagulation rates $K^{a}_{i,j}=K^{b}_{i,j}=k>0$ and constant annihilation rates $J_{i,j}=d>0$, where $k$ and $d$ are not necessarily equal, and investigated its time evolution. They also analyzed the existence or absence of long-time self-similar behaviour. In~\cite{da2012scaling}, F. P.~da Costa et al.\ extended the results of~\cite{laurenccot2010nonuniversal} by relaxing the assumption $K^{a}_{i,j}=K^{b}_{i,j}$ and allowing the two coagulation rates to differ, while still assuming that $K^{a}_{i,j}$ and $K^{b}_{i,j}$ are constant.
 
To the best of our knowledge, the boundedness of the coagulation and annihilation rates has been a standing assumption in studies of the coagulation-annihilation models discussed so far. The primary focus of this work is to demonstrate that the system~\eqref{eqn 1.1}--\eqref{eqn 1.3} is well-posed for unbounded coagulation-annihilation rates, investigate the differentiability of solutions, and analyze their large-time behaviour. The existence result applies, in particular, to coagulation rates of the form $K_{i,j} = \rho_{i} \rho_{j} + \beta_{i,j}$ for all $i, j\ge 1$, where $(\rho_{i})_{i\ge 1}$ is a non-negative real-valued sequence that satisfies either~\eqref{eqn 2.4} or \eqref{eqn 2.5}. In addition, no growth condition on the annihilation rates is required, provided that the initial mass density, as introduced in Section~2, is finite. However, in the case of finite initial number density, as introduced in Section~2, we impose the assumption that $J_{i,j}\le \rho_{i} \rho_{j}$ for all $i,j\ge 1$, where $(\rho_i)_{i\geq 1}$ is the sequence introduced above. Furthermore, we address the case in which coagulation rates grow at most linearly, motivated by~\cite[Theorem 2.3]{laurenccot2002discrete}. The absence of any growth or structural condition on the annihilation rates in \Cref{finite mass theorem} is inspired by \cite{laurenccot2024redner}. Moreover, continuous dependence, uniqueness, differentiability, and large-time behaviour of solutions are discussed under certain  additional assumptions on the coagulation and annihilation rates. 

The paper is structured as follows: \Cref{Section Main results} contains the key results along with the notation used. \Cref{Section Approx systems} is devoted to solving the approximating system and obtaining certain moment estimates. Existence results under different assumptions on the coagulation and annihilation rates are established in \Cref{Section Existence results}. Continuous dependence on the initial data and uniqueness of solutions are addressed in \Cref{Section Uniqueness}. \Cref{Section Differentiability} examines the differentiability of solutions to \eqref{eqn 1.1}--\eqref{eqn 1.3}. Lastly, \Cref{ Large-Time behaviour of solutions} focuses on the large-time behaviour of solutions.
 
	\section{Main results}\label{Section Main results}
	\noindent We begin by introducing the notation used throughout the paper. Let $X_{p}$, $p\in \mathbb{R}$, be the space comprising all real-valued sequences $x=(x_{i})_{i\geq 1}$ such that
	$$\sum_{i = 1}^{\infty} i^p |x_{i}|  < \infty,$$
	with the norm defined by 
	$$ \|x\|_{p} = \sum_{i = 1}^{\infty} i^p |x_{i}|.$$
	$\|x\|_{p}$ is also called the $p$-th moment of $x$. $(X_{p}, \|\cdot\|_{p})$ is a Banach space.
	We also denote the positive cone of $X_{p} $ by $X_{p}^{+}$, that is,
	$$ X^{+}_{p} = \{ x \in X_{p}: x_{i} \geq 0 \ \text{for each} \  i    \geq 1\}.$$
	Note that $\|a\|_{1}$ represents the total density or mass of clusters of type $A$, while $\|a\|_{0}$ denotes the total population of clusters of type $A$.
	
	We now state precisely what is meant by a solution to \eqref{eqn 1.1}--\eqref{eqn 1.3}.
	\begin{definition} \label{Definition}
		For $T\in(0,\infty]$, let $c=(a,b)=\left((a_i)_{i\geq1},(b_i)_{i\geq1}\right)$ be a pair of sequences of non-negative continuous functions. We say that $c$ solves \eqref{eqn 1.1}--\eqref{eqn 1.3} if 
		\begin{enumerate}
			\item For each $i \in \mathbb{N}$,  
			$c_{i}= (a_{i}, b_{i})\in C([0,T))\times C([0,T))$.
			\item For each $i \in \mathbb{N}$ and $t \in [0,T)$,  
			\[
			\sum_{j=1}^{\infty} K_{i,j} a_{j} \in L^{1}(0,t), \quad 
			\sum_{j=1}^{\infty} K_{i,j} b_{j} \in L^{1}(0,t), \quad
			\sum_{j=1}^{\infty} J_{i,j} a_{j} \in L^{1}(0,t), \quad
			\sum_{j=1}^{\infty} J_{i,j} b_{j} \in L^{1}(0,t).
			\]
			\item For each $i\in \mathbb{N}$ and $t\in(0,T)$, the following identities hold:
			\begin{align}\label{eqn 2.1}
				a_{i}(t) &= a_{i}^{0} + \int_{0}^{t} \left( 
				\frac{1}{2} \sum_{j = 1}^{i-1} K_{j,i-j} a_{j}(s) a_{i-j}(s) 
				- a_{i}(s) \sum_{j =1}^{\infty} K_{i,j} a_{j}(s) 
				- a_{i}(s)\sum_{j=1}^{\infty} J_{i,j} b_{j}(s)\right) ds,
			\end{align}
			\begin{align}\label{eqn 2.2}
				b_{i}(t) &= b_{i}^{0} + \int_{0}^{t} \left( 
				\frac{1}{2} \sum_{j = 1}^{i-1} K_{j,i-j} b_{j}(s) b_{i-j}(s) 
				- b_{i}(s) \sum_{j =1}^{\infty} K_{i,j} b_{j}(s) 
				- b_{i}(s)\sum_{j=1}^{\infty} J_{i,j} a_{j}(s)\right) ds.
			\end{align}
		\end{enumerate}
	\end{definition}
 By the preceding definition, it follows immediately that if $\left((a_{i})_{i\ge 1},(b_{i})_{i \ge 1}\right)$ satisfies \eqref{eqn 2.1}--\eqref{eqn 2.2} on $[0,T)$, then $a_{i}$ and $b_{i}$ are absolutely continuous for each $i\ge 1$. Consequently, \eqref{eqn 1.1}--\eqref{eqn 1.2} hold for almost every $t\in [0,T).$

Throughout this paper, we assume the existence of two non-negative real-valued sequences, 
$\rho=(\rho_{i})_{i\geq 1}$ and $(\beta_{i,j})_{i,j\geq 1},$ such that  
\begin{align}\label{eqn 2.3}
	K_{i,j} = \rho_{i} \rho_{j} + \beta_{i,j}, \quad \text{for all } i,j \geq 1.  
\end{align}
In addition, we impose one of the following two sets of assumptions on $(\rho_{i})_{i\geq 1}$ and $(\beta_{i,j})_{i,j\geq 1}$: either 
\begin{align}\label{eqn 2.4}
	\inf_{i \geq 1} \frac{\rho_{i}}{i} = S > 0\quad \text{and}\quad  \beta_{i,j} \leq \kappa \rho_{i} \rho_{j}
	\quad \text{for some}~\kappa>0 \text{ and all } i, j\geq 1,
\end{align}
or
\begin{align}\label{eqn 2.5}
	\lim_{i \to \infty}\frac{\rho_{i}}{i} =0\quad \text{and} \quad \lim_{j \to \infty}\frac{\beta_{i,j}}{j} = 0 \quad \text{for all}~ i \geq 1.
\end{align}
\begin{theorem}\label{finite mass theorem}
Suppose that \eqref{symmetry conditions}, \eqref{eqn 2.3}, and either \eqref{eqn 2.4} or \eqref{eqn 2.5} hold. If $a^{0} = (a_{i}^{0})_{i\geq 1} \in X^{+}_{1}$ and $b^{0} = (b_{i}^{0})_{i\geq 1} \in X^{+}_{1}$, at least one solution $c= (a,b)$ to \eqref{eqn 1.1}--\eqref{eqn 1.3} exists on $[0,\infty)$, as defined in \Cref{Definition}. In addition, we have
		\begin{align}
		(a(t), b(t)) \in X^{+}_{1} \times X^{+}_{1}, \  \  \ 	\|a(t)\|_{1} \leq \|a^{0}\|_{1}, \quad \text{and} \quad \|b(t)\|_{1} \leq \|b^{0}\|_{1} \ \ \text{for each} \ t\in[0,\infty).
		\end{align}
\end{theorem}
To prove \Cref{finite mass theorem}, we adapt the approaches developed in \cite[Theorem 2.2]{laurenccot1999global} and \cite[Theorem 1.2]{laurenccot2024redner}. We first truncate the infinite system of ODEs \eqref{eqn 1.1}--\eqref{eqn 1.3} to a finite system~\eqref{truncated system}, consisting of $4N$ equations for $N \ge 2$. More precisely, we set $K_{i,j} = 0$ and $J_{i,j} = 0$ whenever $\max\{i,j\} > N$, following the truncation procedure in \cite{bak1991finite}, \cite{laurenccot1999global}, and \cite{da1998finite}. The approximating system~\eqref{truncated system} admits a local solution by the Picard--Lindel\"of theorem, which can then be extended to a global solution using suitable moment estimates. We next establish uniform estimates for the solution of the approximating system~\eqref{truncated system} and its derivatives. Combining these bounds with Helly's theorem, we obtain a convergent subsequence. Finally, we show that the limit of this subsequence solves \eqref{eqn 1.1}–\eqref{eqn 1.3}, according to \Cref{Definition}.

\Cref{finite mass theorem} proves the existence of a solution to \eqref{eqn 1.1}--\eqref{eqn 1.3} for initial data satisfying $a^{0}, b^{0}\in X_{1}^{+}$. We now extend this result to the broader class of initial data with finite total number concentration, namely, $a^{0}, b^{0}\in X_{0}^{+}$. 
\begin{theorem}\label{finite particle theorem}
		Suppose that conditions \eqref{symmetry conditions}, \eqref{eqn 2.3}, and \eqref{eqn 2.4} are satisfied. Furthermore, assume that $J_{i,j}\le \rho_{i}\rho_{j}$ for all $i,j\ge1$, where $(\rho_{i})_{i\ge 1}$ is a non-negative real-valued sequence satisfying \eqref{eqn 2.4}. If $a^{0} = (a_{i}^{0})_{i\geq 1} \in X^{+}_{0}$ and $b^{0} = (b_{i}^{0})_{i\geq 1} \in X^{+}_{0}$, at least one solution $c = (a,b)$ to \eqref{eqn 1.1}--\eqref{eqn 1.3} exists on $[0,\infty)$, as defined in \Cref{Definition}.
	 In addition, we have
	\begin{align}\label{eqn 2.8}
		(a(t), b(t)) \in X^{+}_{0} \times X^{+}_{0}, \  \  \ 	\|a(t)\|_{1} < \infty, \quad \text{and} \quad \|b(t)\|_{1} < \infty  \ \ \text{for each} \; t\in (0,\infty).
		\end{align} 
	\end{theorem}
 We prove \Cref{finite particle theorem} along the lines of the argument given in \cite[Theorem 2.2]{laurenccot1999global}. Following the methodology of \Cref{finite mass theorem}, we establish \Cref{finite particle theorem} by employing the truncation procedure used for \Cref{finite mass theorem}. The existence of solutions to the truncated system \eqref{truncated system} is then established via the Picard-Lindel\"of theorem. Next, we establish uniform bounds on the time derivatives of the truncated solutions. These bounds, together with a compact embedding result, allow us to extract a convergent subsequence whose limit constitutes a solution to \eqref{eqn 1.1}--\eqref{eqn 1.3}, according to \Cref{Definition}.
 
The preceding existence results do not apply to the kernel $K_{i,j}=i+j$ for $i,j\ge 1$. We therefore establish the following additional existence result.
\begin{theorem}\label{theorem 2.5}
	Assume that the symmetry conditions \eqref{symmetry conditions} hold and that there exists a constant $\kappa_{1}>0$ such that $ K_{i,j} \leq \ \kappa_{1}(i+j)$ for all $i,j \geq 1$. Consider $a^{0}\in X^{+}_{1} \ \text{and} \ b^{0} \in X^{+}_{1}$. Then a solution $c=(a,b)$ to \eqref{eqn 1.1}--\eqref{eqn 1.3} exists on $[0,\infty)$, as defined in \Cref{Definition}, such that $(a(t),b(t))\in X^{+}_{1}\times X^{+}_{1}$ for every $t\ge 0$. Moreover,
	\begin{align}
	\|a(t)\|_{1} \leq \|a^{0}\|_{1} \quad \text{and} \quad \|b(t)\|_{1} \leq \|b^{0}\|_{1}  \quad \text{for each}\ \  t\geq0.
	\end{align}
\end{theorem}
We finally supplement the above existence results with the following uniqueness result, which requires stronger assumptions on the coagulation and annihilation kernels.
	\begin{theorem}\label{Uniqueness theorem}
		Let $\alpha \in [0,\tfrac{1}{2}]$, and assume that the coagulation and annihilation rates satisfy \eqref{symmetry conditions}, 
$K_{i,j} \leq\kappa_{2}(ij)^{\alpha}$ and $J_{i,j} \leq \kappa_{3}(ij)^{\alpha}$ for all $i,j \geq 1$ and for some constants $\kappa_{2}, \kappa_{3} > 0$.
If $a^{0}, b^{0} \in X_{1}^{+}$, then \eqref{eqn 1.1}--\eqref{eqn 1.2} possesses a unique solution $c=(a,b)$ as defined in \Cref{Definition}, with $(a(t),b(t)) \in X^{+}_{1} \times X^{+}_{1}$ for every $t \geq 0$.
\end{theorem}
\section{Approximating systems}\label{Section Approx systems}
\noindent The existence of a solution to \eqref{eqn 1.1}--\eqref{eqn 1.3} is established by constructing solutions to finite-dimensional truncations of the system and subsequently passing to the limit. More precisely, for $N\ge 2$ and $t>0$, we truncate the system~\eqref{eqn 1.1}--\eqref{eqn 1.3} to obtain a finite-dimensional system consisting of $4N$ equations, given by 
	\begin{align}
		\label{truncated system}
		\begin{cases}
			\displaystyle \frac{d a_i^{N}}{dt} = \frac{1}{2} \sum_{j=1}^{i-1} K_{j,i-j}^{N}~a_{j}^{N} a_{i-j}^{N}
			- \sum_{j=1}^{N} K_{i,j}^{N}~a_{i}^{N} a_{j}^{N}
			- \sum_{j=1}^{N} J_{i,j}^{N}~a_i^{N} b_{j}^{N}, &  \ \    1 \leq i \leq N, \\[1em]
			\displaystyle \frac{d a_i^{N}}{dt} = \frac{1}{2} \sum_{j=i-N}^{N} K_{j,i-j}^{N}~a_{j}^{N} a_{i-j}^{N}, & \  \  N+1 \leq i \leq 2N, \\[1em]
			\displaystyle \frac{d b_i^{N}}{dt} = \frac{1}{2} \sum_{j=1}^{i-1} K_{j,i-j}^{N}~b_{j}^{N} b_{i-j}^{N}
			- \sum_{j=1}^{N} K_{i,j}^{N}~b_{i}^{N} b_{j}^{N}
			- \sum_{j=1}^{N} J_{i,j}^{N}~b_i^{N} a_{j}^{N}, & \   \   1 \leq i \leq N, \\[1em]
			\displaystyle \frac{d b_i^{N}}{dt} = \frac{1}{2} \sum_{j=i-N}^{N} K_{j,i-j}^{N}~b_{j}^{N} b_{i-j}^{N}, & \   \   N+1 \leq i \leq 2N, \\[1em]
			a_i^{N}(0) = a_i^0, \quad b_i^{N}(0) = b_i^0, & \  \   1\leq i\leq 2N,
		\end{cases}
	\end{align}
    where
    \begin{align*}
        K_{i,j}^{N}=\begin{cases}
             K_{i,j}, &\quad \max \{i, j\}\leq N \\
             0, & \quad \max \{i, j\}> N
        \end{cases}
        \quad \text{and}\quad J_{i,j}^{N}=\begin{cases}
             J_{i,j}, &\quad \max \{i, j\}\leq N \\
             0, & \quad \max \{i, j\}> N.
        \end{cases}
    \end{align*}
    
	The proposition below shows that \eqref{truncated system} admits a unique local non-negative solution. 
    \begin{prop}\label{Locexistenceforapproximatesystem}
For $N \ge 2$, the truncated system~\eqref{truncated system} admits a unique local non-negative solution $\left((a^{N}_{i})_{1\le i \le 2N}, (b^{N}_{i})_{1\le i \le 2N}\right)$, defined on $[0,T_{c})$ for some $T_{c}\in (0,\infty].$
	\end{prop}
	\begin{proof}\Cref{Locexistenceforapproximatesystem}~is an immediate consequence of the Picard--Lindel\"of theorem for systems of ordinary differential equations. For completeness, we present the main  steps of the proof.  For  $k = 1, 2, 3, 4$ and  $i = 1, 2, \ldots, 2N$, with $N \ge 2$, define 
	$f_{i,k}:\mathbb{R}^{2N}\times\mathbb{R}^{2N}\rightarrow \mathbb{R}$ as follows:
		\begin{align*}
			f_{i, 1}(x^{N},y^{N}) = & {\frac{1}{2} \ \sum_{j=1}^{i-1} K_{j,i-j}^{N}~x_{j}^N x_{i-j}^N} {- \sum_{j=1}^{N}  K_{i,j}^{N}~ x_{i}^N x_{j}^N   - \sum_{j=1}^{N} J_{i,j}^{N}~x_i^N y_{j}^N,}\ 1 \leq i \leq N, \\
			f_{i, 2}(x^{N}, y^{N}) = & \frac{1}{2}\ \sum_{j=i-N}^{N}K_{j,i-j}^{N}~x_{j}^N x_{i-j}^N,
			\  \ N+1\leq i \leq 2N, \\
			f_{i, 3}(x^{N},y^{N}) =& {\frac{1}{2}\ \sum_{j=1}^{i-1} K_{j,i-j}^{N}~y_{j}^N y_{i-j}^N} 
			- \sum_{j=1}^{N}  K_{i,j}^{N}~ y_{i}^N y_{j}^N - \sum_{j=1}^{N} J_{i,j}^{N}~y_i^N x_{j}^N,\ \ 1 \leq i \leq N,\\
               \text{and}\;\;f_{i, 4}(x^{N},y^{N})  = &\frac{1}{2}\ \sum_{j=i-N}^{N} K_{j,i-j}^{N}~y_{j}^N y_{i-j}^N,\;\;N+1\leq i \leq 2N,\\
\text{for all}~~({x^{N},y^{N})} \in \mathbb{R}^{2N}\times\mathbb{R}^{2N},&  ~~  \text{where} ~ x^{N} = (x_{1}^{N}, x_{2}^{N},\ldots,x^{N}_{2N})~\text{and}~y^{N} = (y_{1}^{N}, y_{2}^{N},\ldots,y_{2N}^{N}).
\end{align*}
Now, \eqref{truncated system} can be written as		
\begin{align*}
\frac{da_{i}^{N}}{dt}
&=
\begin{cases}
f_{i,1}\left(a^{N},b^{N}\right), & 1\le i\le N,\\
f_{i,2}\left(a^{N},b^{N}\right), & N+1\le i\le 2N
\end{cases}
\\[1ex]
\text{and}\quad\frac{db_{i}^{N}}{dt}
&=
\begin{cases}
f_{i,3}\left(a^{N},b^{N}\right), & 1\le i\le N,\\
f_{i,4}\left(a^{N},b^{N}\right), & N+1\le i\le 2N,
\end{cases}
\end{align*}
        where
      \begin{align*}a^{N} = (a_{1}^{N}, a_{2}^{N},\ldots,a^{N}_{2N})\quad \text{and}\quad b^{N} = (b_{1}^{N}, b_{2}^{N},\ldots,b_{2N}^{N}). \end{align*}
		Since the $f_{i,k}$ for $k =1, 2, 3, 4$ are polynomials, they are locally Lipschitz continuous. Hence, the Picard--Lindel\"of theorem provides a unique solution $\left((a_i^N)_{i=1}^{2N}, (b_i^N)_{i=1}^{2N}\right)$
 on a maximal time interval ~$I = [0, T_{c})$, where $T_{c}\in (0,\infty].$ Next, we proceed to prove that the solution obtained for \eqref{truncated system} is non-negative. For this purpose, we adopt the argument developed in \cite[Proposition 2.3]{MR3135638}. Suppose there exists $\gamma \in I$ such that $a_{i}^{N}(\gamma)$ $\geq$ 0 for all $i = 1,\ldots, 2N,$ and such that $a_{m_{o}}^{N}(\gamma )= 0$, while $a_{m_{o}}^{N}(t) < 0$ for $t \in (\gamma,\gamma+\eta)$ where $m_{o} \in \{1, 2, \ldots, 2N\}$ and  $0<\eta< T_{c}-\gamma$. For $\epsilon > 0$, we consider the initial value problem  
    \begin{align}
		\label{perturbedfinitesystem}
		\begin{cases} 
			\displaystyle \dot a_i^{\epsilon,N} = \frac{1}{2} \sum_{j=1}^{i-1} K_{j,i-j}^{N}~a_{j}^{\epsilon,N} a_{i-j}^{\epsilon,N}
			- \sum_{j=1}^{N} K_{i,j}^{N}~a_{i}^{\epsilon,N} a_{j}^{\epsilon,N}
			- \sum_{j=1}^{N} J_{i,j}^{N}~a_i^{\epsilon,N} b_{j}^{\epsilon,N} + \epsilon, &  \ \    1 \leq i \leq N, \\
			\displaystyle \dot a_i^{\epsilon,N} = \frac{1}{2} \sum_{j=i-N}^{N} K_{j,i-j}^{N}~a_{j}^{\epsilon,N} a_{i-j}^{\epsilon,N} + \epsilon, & \  \  N+1 \leq i \leq 2N, \\
			\displaystyle \dot b_i^{\epsilon,N} = \frac{1}{2} \sum_{j=1}^{i-1} K_{j,i-j}^{N}~b_{j}^{\epsilon,N} b_{i-j}^{\epsilon,N}
			- \sum_{j=1}^{N} K_{i,j}^{N}~b_{i}^{\epsilon,N} b_{j}^{\epsilon,N}
			- \sum_{j=1}^{N} J_{i,j}^{N}~b_i^{\epsilon,N} a_{j}^{\epsilon,N} + \epsilon, & \   \   1 \leq i \leq N, \\
			\displaystyle \dot b_i^{\epsilon,N} = \frac{1}{2} \sum_{j=i-N}^{N} K_{j,i-j}^{N}~b_{j}^{\epsilon,N} b_{i-j}^{\epsilon,N} + \epsilon, & \  \  N+1 \leq i \leq 2N, \\
		a_{i}^{\epsilon,N}(\gamma) = a_{i}^{N}(\gamma), \quad b_{i}^{\epsilon,N}(\gamma) = b_i^{N}(\gamma), & \  \   1\leq i\leq 2N.
		\end{cases}
	\end{align}
The Picard--Lindel\"of theorem ensures that \eqref{perturbedfinitesystem} possesses a unique solution $\left((a^{\epsilon,N}_{i})_{i=1}^{2N},(b^{\epsilon,N}_{i})_{i=1}^{2N}\right)$ on a maximal time interval $I_{\epsilon}=[\gamma,\gamma + T_{\epsilon,c})$, for some $T_{\epsilon,c}\in(0,\infty].$ From the first two equations of \eqref{perturbedfinitesystem}, we obtain $\dot a_{m_{o}}^{\epsilon,N}(\gamma)>0$, which implies that  $a_{m_{o}}^{\epsilon,N}{(t)}>0$ for all $t \in(\gamma,\gamma+\delta)$, where  $0<\delta<\min\{T_{\epsilon,c},\eta\}$. 
    Moreover, the solution to \eqref{perturbedfinitesystem} depends continuously on the data. Hence, as $\epsilon$ decreases to $0$, the solution of \eqref{perturbedfinitesystem} converges to the solution of \eqref{truncated system}; that is, $$a_{m_{0}}^{\epsilon,N}(t) \xrightarrow{\epsilon\downarrow 0} a_{m_{0}}^{N}(t) , \text{ for}~ t \in [\gamma, \gamma+ \delta),$$ which yields $a_{m_{0}}^{N}(t) \geq 0$ for $t \in [\gamma, \gamma+ \delta)$. This is in contradiction with the assumption that $a_{m_{0}}^{N}(t) < 0$  for $t \in (\gamma,\gamma+\eta)$. 
Hence, $a_{i}^{N}(t) \geq 0$ for all $t \in I$ and $1\leq i\leq 2N$. Similarly, we can show that $b_{i}^{N}(t) \geq 0$  for all $t \in I$ and $1\leq i\leq 2N$. The proof of~\Cref{Locexistenceforapproximatesystem} is thereby concluded.
    \end{proof}
	
	Now, we invoke the following lemma inspired by \cite[Lemma 3.1]{laurenccot1999global}, which will be used in the subsequent analysis to derive the required moment estimates.
	\begin{lemma} \label{moment estimate lemma}
Assume that \eqref{symmetry conditions} is satisfied. For $t \in [0,T_{c})$, $\sigma \in [0,t],$ and any finite family of real numbers $(h_{i})_{1\leq i \leq 2N}$, the following identities hold:
	\begin{align}
		\sum_{i=1}^{2N} h_{i} a_{i}^{N}(t) - \sum_{i=1}^{2N} h_{i} a_{i}^{N}(\sigma) = &  \ \frac{1}{2} \int_{\sigma}^{t} \sum_{i=1}^{N} \sum_{j=1}^{N}(h_{i+j} - h_{i}- h_{j})K_{i,j} a_{i}^{N}(s) a_{j}^{N}(s) ds \nonumber\\
		& -  \int_{\sigma}^{t} \sum_{i=1}^{N} \sum_{j=1}^{N} h_{i}  J_{i,j} a_{i}^{N}(s)  b_{j}^{N}(s) ds\label{moment estimate lemma for A}\\
	\text{and} \quad
		\sum_{i=1}^{2N} h_{i} b_{i}^{N}(t) - \sum_{i=1}^{2N} h_{i} b_{i}^{N}(\sigma) = &  \ \frac{1}{2} \int_{\sigma}^{t} \sum_{i=1}^{N} \sum_{j=1}^{N}(h_{i+j} - h_{i}- h_{j})K_{i,j} b_{i}^{N}(s) b_{j}^{N}(s) ds \nonumber\\
		& -  \int_{\sigma}^{t} \sum_{i=1}^{N} \sum_{j=1}^{N} h_{i}  J_{i,j} b_{i}^{N}(s)  a_{j}^{N}(s) ds.\label{moment estimate lemma for B}
	\end{align}
    \end{lemma}
		By substituting appropriate values for $(h_{i})_{1 \leq i \leq 2N}$ and applying some inequalities, we derive the following lemma. In particular, the next lemma establishes moment bounds needed for compactness arguments and convergence analysis.
	\begin{lemma}\label{lemma 3.3}
	Let the assumptions~\eqref{symmetry conditions} and \eqref{eqn 2.3} be satisfied. Then, for any $H \in \{1,\ldots, N\}$, $t \in [0,T_{c}),$ and $\sigma \in [0, t]$, we have
\begin{align}
	\sum_{i=1}^{N} ig_{i,j}^{N}(t) \leq \sum_{i=1}^{N} ig_{i,j}^{N}(\sigma)& \leq\sum_{i=1}^{N} ig_{i,j}^{0}, \ \text{for} \  j = 1,2  ,\label{first moment bound}\\
	\sum_{i=1}^{2N} g_{i,j}^{N}(t) + \frac{1}{2} \int_{0}^{t} \bigg|\sum_{i=1}^{N} \rho_{i} g_{i,j}^{N}(s) \bigg| ^2\, ds &\leq \sum_{i =1}^{2N} g_{i,j}^{0},  \ \text{for} \  j = 1,2  ,\label{zeroth moment bound}\\
\int_{\sigma}^{t} \sum_{i=1}^{N} \sum_{j=1}^{N} i^{\mu}  J_{i,j} a_{i}^{N}(s) b_{j}^{N}(s)\, ds &\leq \sum_{i=1}^{2N} i^{\mu} a^{N}_{i}(\sigma)\le \sum_{i=1}^{2N} i^{\mu} a^{0}_{i},  \ \text{for} \  \mu = 0,1,\label{double sum first moment for A}
\end{align}
\begin{align}
\int_{\sigma}^{t} \sum_{i=1}^{N} \sum_{j=1}^{N} i^{\mu}  J_{i,j} b_{i}^{N}(s) a_{j}^{N}(s)\,ds & \leq \sum_{i=1}^{2N} i^{\mu} b_{i}^{N}(\sigma) \le  \sum_{i=1}^{2N} i^{\mu} b_{i}^{0},  \ \text{for} \  \mu = 0,1,\label{double sum first moment for B}\\
	\int_{\sigma}^{t} \bigg| \sum_{i=H}^{N}\rho_{i} g_{i,j}^{N}(s)\bigg|^2\, ds  &\leq 4 \left(  \sum_{i=1}^{N}  i^{\frac{1}{2}}  g_{i,j}^{N}(\sigma)\right ) H^{-\frac{1}{2}}\ \text{for} \  j = 1,2,  \label{half moment bound}
    \end{align}
where $g_{i,1}^{N} = a_{i}^{N}$, $g_{i,2}^{N} = b_{i}^{N}$, $g_{i,1}^{0} = a_{i}^{0}$, and   $g_{i,2}^{0} = b_{i}^{0}$ for each $i \in \{1,\ldots, 2N\}$.
\end{lemma}
	\begin{proof}
The proof of \eqref{first moment bound}, together with those of \eqref{double sum first moment for A} and \eqref{double sum first moment for B} for $\mu=1$, follows by setting 
\begin{align*}
        h_{i} =\begin{cases}
             i, &\quad \text{if}~ i \in \{1,\ldots,N\},\\
             0, & \quad\text{if}~ i \in \{N+1,\ldots, 2N\}
        \end{cases}
        \end{align*}
in \eqref{moment estimate lemma for A} and \eqref{moment estimate lemma for B}, along with the non-negativity of $\left((a^{N}_{i})_{1\leq i \leq 2N}, (b^{N}_{i})_{1 \leq i \leq 2N}\right)$. Next, by taking $h_{i} = 1$ for each $i \in \{1,\ldots,2N\}$ and using \eqref{eqn 2.3} in \eqref{moment estimate lemma for A} and \eqref{moment estimate lemma for B}, we deduce \eqref{zeroth moment bound}, together with \eqref{double sum first moment for A} and \eqref{double sum first moment for B} for $\mu=0$.
Now, in order to prove \eqref{half moment bound}, we consider
 \begin{align*}
        h_{i} =\begin{cases}
             i^{1/2}, &\quad \text{if}~ i \in \{1,\ldots,N\},\\
             0, & \quad\text{if}~ i \in \{N+1,\ldots, 2N\}
        \end{cases}
        \end{align*}
in \Cref{moment estimate lemma}. Observe that for all \(1 \leq i,j \leq N\),  \begin{align}\label{inequality1}
\frac{1}{2}\min\{i,j\}^{1/2} \leq i^{1/2}+j^{1/2}-(i+j)^{1/2} \leq h_{i}+h_{j}-h_{i+j}.
\end{align} 
For \(t \in [0, T_{c})\) and \(\sigma \in [0, t]\), applying \eqref{inequality1} in \eqref{moment estimate lemma for A} and \eqref{moment estimate lemma for B}, together with \eqref{eqn 2.3}, yields
\[
\int_{\sigma}^{t} \sum_{i=1}^{N}\sum_{j=1}^{N} 
\min\{i,j\}^{1/2}\, \rho_i \rho_j \, a_i^N(s)\, a_j^N(s)\, ds 
\leq 4 \sum_{i=1}^{N} i^{1/2}\, a_i^N(\sigma).
\]
We obtain, for any \(H \in \{1,\ldots,N\}\),  
\[
H^{1/2} \int_{\sigma}^{t} \sum_{i=H}^{N}\sum_{j=H}^{N} 
\rho_i \rho_j\, a_i^N(s)\, a_j^N(s)\, ds \leq
\int_{\sigma}^{t} \sum_{i=1}^{N}\sum_{j=1}^{N} 
\min\{i,j\}^{1/2}\, \rho_i \rho_j\, a_i^N(s)\, a_j^N(s)\, ds
\leq 4 \sum_{i=1}^{N} i^{1/2}\, a_i^N(\sigma).
\]  
Therefore, \eqref{half moment bound} holds for $j=1$. Similarly, we can deduce \eqref{half moment bound} for $j=2.$
\end{proof}
\begin{remark}
By applying \eqref{zeroth moment bound}, we obtain the uniform estimates:
\[
0 \leq a_{i}^{N}(t) \leq \|a^{0}\|_{0}, \quad  
0 \leq b_{i}^{N}(t) \leq \|b^{0}\|_{0},
\]
for all \(i \in \{1, \ldots, 2N\}\) and \(t \in [0, T_{c})\). These bounds exclude the possibility of finite-time blow-up of the solution $\left((a^{N}_{i})_{1\leq i \leq 2N}, (b^{N}_{i})_{1 \leq i \leq 2N}\right)$ and, by the maximal interval of existence argument~\cite[Corollary~2.16]{teschl2012ordinary}, yield~\(T_{c} = \infty\).
\end{remark} 
The next lemma provides a bound on the first moment of $g_{i,j}^{N}(t)$
in terms of the initial zeroth moment and the time variable $t$.
\begin{lemma}Suppose that \eqref{eqn 2.3} and \eqref{eqn 2.4} hold. Then, for $t>0$ and $j=1,2,$
 \begin{align}
\quad \sum_{i=1}^{N} i  g_{i,j}^{N}(t) \leq \frac{2}{S}\left(\sum_{i=1}^{2N} g_{i,j}^{0}\right)^{1/2}t^{-\frac{1}{2}},\label{S inequality}
\end{align}where $g_{i,1}^{N} = a_{i}^{N}$, $g_{i,2}^{N} = b_{i}^{N}$, $g_{i,1}^{0} = a_{i}^{0}$, and   $g_{i,2}^{0} = b_{i}^{0}$ for each $i \in \{1,\ldots, 2N\}$, and $S$ is given by \eqref{eqn 2.4}.
\begin{proof}
Combining \eqref{eqn 2.4} with \eqref{zeroth moment bound} for $j=1$ gives, for every $t>0$,
\begin{equation}\label{S^2 inequality}
\int_{0}^{t} \left| \sum_{i=1}^{N} i a_i^N(s) \right|^{2} ds 
\leq \frac{2}{S^{2}} \sum_{i=1}^{2N} a_i^{0}.
\end{equation}
Since the function
$$ s \mapsto\sum_{i=1}^{N} i a_i^N(s) $$
is non-increasing by \eqref{first moment bound}, \eqref{S^2 inequality} implies that 
$$
t \left| \sum_{i=1}^{N} i a_i^N(t) \right|^{2} \leq \frac{2}{S^{2}} \sum_{i=1}^{2N} a_i^{0},
$$
which yields \eqref{S inequality} for $j=1$. Similarly, we can deduce \eqref{S inequality} for $j=2$.
\end{proof}
\end{lemma}
		\section{Existence}\label{Section Existence results}
 \noindent Prior to proving \Cref{finite mass theorem} under the assumptions \eqref{eqn 2.3} and \eqref{eqn 2.4} on the coagulation rates, we introduce some notation. In view of \eqref{truncated system}, for each $i \in \{1, \dots, N\}$ and $t \in (0, \infty)$, we have
	\begin{align}\label{decomposition}
		\frac{d a_{i}^{N}}{dt} &= Q_{i}(a^N) + P_{i}(a^N,b^N), \quad 
		\frac{d b_{i}^{N}}{dt} = Q_{i}(b^N)+ R_{i}(a^N,b^N),
	\end{align}
	where
	\begin{align}
		Q_{i}(a^N) &= \frac{1}{2} \sum_{j=1}^{i-1} K^{N}_{j,i-j}\, a_{j}^{N} a_{i-j}^{N} 
		- \sum_{j=1}^{N} K^{N}_{i,j}\, a_{i}^{N} a_{j}^{N}, \quad	P_{i}(a^N,b^N) = - \sum_{j=1}^{N} J_{i,j}^{N}\, a_{i}^{N} b_{j}^{N}, \\
		 \ Q_{i}(b^N) &= \frac{1}{2} \sum_{j=1}^{i-1} K^{N}_{j,i-j}\, b_{j}^{N} b_{i-j}^{N} 
		- \sum_{j=1}^{N} K^{N}_{i,j}\, b_{i}^{N} b_{j}^{N}, \quad R_{i}(a^N,b^N) = - \sum_{j=1}^{N} J^{N}_{i,j}\, b_{i}^{N} a_{j}^{N},\\
        \text{and}& \qquad a^{N} = (a_{1}^{N}, a_{2}^{N},\ldots,a^{N}_{N}),\quad b^{N} = (b_{1}^{N}, b_{2}^{N},\ldots,b_{N}^{N}). 
	\end{align}
   The following lemma establishes the boundedness of $\bigg(\frac{da_{i}^{N}}{dt}\bigg)_{N\ge i}$and $\bigg(\frac{db_{i}^{N}}{dt}\bigg)_{N\ge i}$  in $L^{1}(0,T)$ for every $T>0$.
    \begin{lemma}\label{conc3 bounds*}
Suppose that \eqref{symmetry conditions}, \eqref{eqn 2.3}, and \eqref{eqn 2.4} hold, and that $a^{0}, b^{0}\in X_{1}^{+}$. Then, for all $i\geq1,$ $ N\geq i,$ and $T>0$, the following estimates hold:
\begin{align}
	a_{i}^{N}(t) \leq &\| a^{0} \|_{0} \quad\text{and} \quad b_{i}^{N}(t)  \leq \| b^{0} \|_{0}, \ 	\text {for all} \  t \in [0,T] \label{conc2 bounds*},\\
    \left\| \frac{da_{i}^{N}}{dt} \right\|_{L^{1}(0,T)} \leq & \ T^{1/2} \  \Bigg\{ T^{1/2} \left( \sum_{j=1}^{i-1} K_{j,i-j} \right) \|a^{0}\|_{0}^2 + 2 (1+\kappa) \rho_{i} \|a^{0}\|_{0}^{3/2}\Bigg\} + \|a^{0}\|_{0},\\
  \text{and} ~~  \left\| \frac{db_{i}^{N}}{dt} \right\|_{L^{1}(0,T)} \leq & \ T^{1/2} \  \Bigg\{ T^{1/2} \left( \sum_{j=1}^{i-1} K_{j,i-j} \right) \|b^{0}\|_{0}^2 + 2 (1+\kappa) \rho_{i} \|b^{0}\|_{0}^{3/2}\Bigg\} + \|b^{0}\|_{0}.
    \end{align}
    \end{lemma}
    \begin{proof}
   \eqref{conc2 bounds*} follows immediately from \eqref{zeroth moment bound}. Moreover, from \eqref{eqn 2.3}, \eqref{eqn 2.4}, and \eqref{zeroth moment bound}, we infer that
\begin{align}\label{bound for L^2 norm}
			\|Q_{i}(a^N)\|_{L^{2}(0,T)}	 \leq & \ T^{1/2} \left( \sum_{j=1}^{i-1} K_{j,i-j} \right) \|a^{0}\|_{0}^2 + 2 (1+\kappa)\rho_{i} \|a^{0}\|_{0}^{3/2}.
            \end{align}
 Furthermore, let $(\Omega,\mathcal{F},\nu)$ be a finite measure space and let $0 < r < s < \infty$. Then, for every  $g \in L^{s}(\nu)$, we have
\begin{align}\label{inequality}
\|g\|_{L^{r}} \le \nu(\Omega)^{(s - r)/ (rs)} \|g\|_{L^{s}},
\end{align}
 see \cite[Section~7.10, pp.~197]{axler2020measure}.
Now, choosing $r = 1$, $s = 2$, $g = Q_{i}(a^N)$, and $X = (0,T)$ in \eqref{inequality}, and invoking \eqref{bound for L^2 norm}, we obtain
\begin{align}
			\|Q_{i}(a^{N})\|_{L^{1}(0,T)}	 \leq & \ T^{1/2} \  \Bigg\{ T^{1/2} \left( \sum_{j=1}^{i-1} K_{j,i-j} \right) \|a^{0}\|_{0}^2 + 2 (1+\kappa) \rho_{i} \|a^{0}\|_{0}^{3/2}\Bigg\}.\label{derivative bound for Q_{i}(a)*}
            \end{align}
           By an analogous argument, we get
            \begin{align}
			\|Q_{i}(b^{N})\|_{L^{1}(0,T)}	 \leq&  \ T^{1/2} \  \Bigg\{ T^{1/2} \left( \sum_{j=1}^{i-1} K_{j,i-j} \right) \|b^{0}\|_{0}^2 + 2 (1+\kappa) \rho_{i} \|b^{0}\|_{0}^{3/2}\Bigg\}.\label{derivative bound for Q_{i}(b)*}
		\end{align}
 Moreover, in view of \eqref{double sum first moment for A} and \eqref{double sum first moment for B} for $\mu=0$, we infer that 
		\begin{align}\label{derivative bound for e*}
		\| P_{i}(a^{N},b^{N})\|_{L^1(0,T)} \leq  \|a^{0}\|_{0}\quad \text{and} \quad \| R_{i}(a^{N},b^{N})\|_{L^1(0,T)} \leq  \|b^{0}\|_{0}.
		\end{align}
Combining \eqref{derivative bound for Q_{i}(a)*}, \eqref{derivative bound for Q_{i}(b)*}, and \eqref{derivative bound for e*}, together with \eqref{decomposition}, we conclude that 
\begin{align*}
		\left\|\frac{da_{i}^{N}}{dt} \right\|_{L^{1}(0,T)} \leq \ T^{1/2} \  \Bigg\{ T^{1/2} \left( \sum_{j=1}^{i-1} K_{j,i-j} \right) \|a^{0}\|_{0}^2 + 2 (1+\kappa) \rho_{i}\|a^{0}\|_{0}^{3/2}\Bigg\} + \|a^{0}\|_{0}\\
        \text{and}\;\;
\left\|\frac{db_{i}^{N}}{dt} \right\|_{L^{1}(0,T)} \leq \ T^{1/2} \  \Bigg\{ T^{1/2} \left( \sum_{j=1}^{i-1} K_{j,i-j} \right) \|b^{0}\|_{0}^2 + 2 (1+\kappa) \rho_{i}\|b^{0}\|_{0}^{3/2}\Bigg\} + \|b^{0}\|_{0}.
\end{align*}
\end{proof}
Now, by \Cref{conc3 bounds*}, it follows that the sequence $(a^{N}_{i})_{N \geq i}$ is uniformly bounded in $W^{1,1}(0,T)$. 
Hence, by Helly's theorem \cite[pp.~372--374]{kolmogorov1975introductory}
and a diagonal argument, we can extract a subsequence
$(a_i^{N_k})_{N_k\ge i}$ of $(a_i^N)_{N\ge i}$ and obtain a
non-negative sequence $a=(a_i)_{i\ge1}$ for which
\begin{align}\label{limit of seq a*}
    a_i^{N_k}(t) \to a_i(t)
    \qquad \text{as } k\to\infty,
\end{align}
for each $i\ge1$ and $t\ge0$.
Similarly, since $(b_{i}^{N_{k}})_{N_{k}\ge i}$ is uniformly bounded in $W^{1,1}(0,T)$, we can extract a further subsequence $(b_{i}^{N_{k_{p}}})_{N_{k_{p}}\ge i}$ of  $(b_{i}^{N_{k}})_{N_{k}\ge i}$ and obtain a non-negative sequence $b = (b_{i})_{i \geq 1} $ for which
\begin{align}\label{limit of relabeled seq b*}
	b_{i}^{N_{k_{p}}}(t) \to b_{i}(t)~~ \text{as}~~p\to \infty~~\text{for each}~i \geq 1 ~\text{and}~t\ge0.
    \end{align}
Moreover, along the same subsequence, we also have
\begin{align}\label{limit of relabeled seq a*}
	a_{i}^{N_{k_{p}}}(t) \to a_{i}(t)~~ \text{as}~~p\to \infty~~\text{for each}~i \geq 1 ~\text{and}~t\ge0.
\end{align}
For $t\in[0,\infty),~ l\ge 1,$ and $N_{k_{p}}\ge l,$ the estimate in \eqref{first moment bound} gives
\[
\sum_{i=1}^{l} i\, a_i^{N_{k_p}}(t) \le \|a^0\|_{1},
\qquad
\sum_{i=1}^{l} i\, b_i^{N_{k_p}}(t) \le \|b^0\|_{1}.
\]
Using \eqref{limit of relabeled seq a*} and \eqref{limit of relabeled seq b*}, we let $p \to \infty$ in the above inequalities, yielding  
\[
\sum_{i=1}^{l} i\, a_{i}(t) \leq \|a^{0}\|_{1}, 
\qquad 
\sum_{i=1}^{l} i\, b_{i}(t) \leq \|b^{0}\|_{1}, 
\quad t \in [0, \infty).
\]
Since these estimates hold for $l \geq 1$, we obtain  
\begin{align}\label{solfirstmomentestimates}
\sum_{i=1}^{\infty} i\, a_{i}(t) \leq \|a^{0}\|_{1}, 
\qquad 
\sum_{i=1}^{\infty} i\, b_{i}(t) \leq \|b^{0}\|_{1}, 
\quad t \in [0, \infty).
\end{align}  
Hence,  
\[
c(t) = \left((a_{i}(t))_{i \geq 1}, (b_{i}(t))_{i \geq 1}\right) \in X_{1}^{+} \times X_{1}^{+}, 
\quad t \in [0, \infty).
\]
Applying \eqref{eqn 2.3}, \eqref{eqn 2.4}, \eqref{first moment bound}, \eqref{double sum first moment for A}--\eqref{half moment bound}, \eqref{limit of relabeled seq b*}, \eqref{limit of relabeled seq a*}, and Fatou's lemma, we obtain, for each $i\ge 1$ and $T>0$,
\begin{align}\label{L1L2integrability}
\sum_{j=1}^{\infty} K_{i,j}\,a_j,\;
\sum_{j=1}^{\infty} K_{i,j}\,b_j\in L^2(0,T)
\,\text{and}\,
\sum_{j=1}^{\infty} J_{i,j}\,a_i b_j,\;
\sum_{j=1}^{\infty} J_{i,j}\,b_i a_j\in L^1(0,T).
\end{align}
Next, we claim that $c=\left((a_{i})_{i \geq 1},(b_{i})_{i\geq 1}\right)$ satisfies \eqref{eqn 2.1}--\eqref{eqn 2.2}. For that purpose, we first state the lemma given below, which will play a key role in the convergence analysis; see \cite[Theorem 2.2]{laurenccot1999global}.
\begin{lemma}\label{mainconvolemmaforthmi}
Assume that the conditions \eqref{symmetry conditions}, \eqref{eqn 2.3}, and \eqref{eqn 2.4} are satisfied, and that $a^{0},b^{0} \in X_{1}^{+}$. Then, for all $1\leq i \leq N_{k_p}$ and every $T>0 $, we have 
\begin{align}\label{tail of C series for A}
	\lim_{p \to \infty} \left\| \sum_{j=1}^{N_{k_p}} K_{i,j}\,a_{j}^{N_{k_p}} - \sum_{j=1}^{\infty} K_{i,j}\,a_{j} \right\|_{L^{2}(0,T)} &= 0, \\
	\lim_{p \to \infty} \left\| \sum_{j=1}^{N_{k_p}} K_{i,j}\,b_{j}^{N_{k_p}} - \sum_{j=1}^{\infty} K_{i,j}\,b_{j} \right\|_{L^{2}(0,T)} &= 0, \label{tail of C series for B}\\
		\lim_{p \to \infty}\bigg\|\sum_{j=1}^{N_{k_{p}}} J_{i,j} a_{i}^{N_{k_{p}}} b_{j}^{N_{k_{p}}} - \sum_{j=1}^{\infty} J_{i,j} a_{i} b_{j}\bigg\|_{L^1(0,T)}  &= 0, \label{convoannihilationA}\\
        \text{and}\;\;
		\lim_{p \to \infty} \bigg\|\sum_{j=1}^{N_{k_{p}}} J_{i,j} b_{i}^{N_{k_{p}}} a_{j}^{N_{k_{p}}} -\sum_{j=1}^{\infty} J_{i,j} b_{i} a_{j}\bigg\|_{L^1(0,T)}  &= 0.  \label{convoannihilationB}
\end{align}
\end{lemma}
\begin{proof}
All the terms in~\eqref{tail of C series for A}--\eqref{convoannihilationB} are well defined due to \eqref{eqn 2.3}, \eqref{eqn 2.4}, \eqref{double sum first moment for A}, \eqref{double sum first moment for B}, \eqref{first moment bound}, \eqref{half moment bound}, and \eqref{L1L2integrability}. The convergence in \eqref{tail of C series for A} and \eqref{tail of C series for B} follow from \eqref{eqn 2.3}, \eqref{eqn 2.4}, \eqref{first moment bound}, \eqref{zeroth moment bound}, \eqref{half moment bound}, \eqref{limit of relabeled seq a*}, \eqref{limit of relabeled seq b*}, and Lebesgue's dominated convergence theorem (DCT); see \cite[Section~4]{laurenccot1999global} and \cite[Section~2.2]{laurenccot2024redner} for similar detailed proofs.
To establish \eqref{convoannihilationA}, it remains only to show that  
\begin{align}\label{tail of A series for A}
	\limsup_{p \to \infty} \bigg\| \sum_{j=1}^{N_{k_p}} J_{i,j}\,a_{i}^{N_{k_p}}(s)\,b_{j}^{N_{k_p}}(s) 
	- \sum_{j=1}^{\infty} J_{i,j}\,a_{i}(s)\,b_{j}(s) \bigg\|_{L^1(0,T)} = 0
\end{align}
for every $i \ge 1$. Fix a positive integer $l$ such that $2 \leq l < N_{k_{p}}< \infty$, and let $T\in (0, \infty)$. Then, using the triangle inequality and the non-negativity of the solutions $\left((a_{i}^{N_{k_{p}}})_{1\leq i\leq N_{k_{p}}}, (b_{i}^{N_{k_{p}}})_{1\leq i\leq N_{k_{p}}}\right)$ and $\left((a_{i})_{i\ge 1}, (b_{i})_{i\ge 1}\right)$, we deduce that
		\begin{align}
			& \int_{0}^{T} \left|\sum_{j=1}^{N_{k_{p}}}  J_{i,j}  a_{i}^{N_{k_{p}}} b_{j}^{N_{k_{p}}} - \sum_{j=1}^{\infty} J_{i,j} a_{i} b_{j}\right| ds  \nonumber \\
            &\leq  \  \ \int_{0}^{T} \sum_{j=1}^{l-1} J_{i,j} \left|a_{i}^{N_{k_{p}}} b_{j}^{N_{k_{p}}} - a_{i} b_{j}\right| ds  +  \int_{0}^{T}		 \sum_{j=l}^{N_{k_{p}}} J_{i,j} a_{i}^{N_{k_{p}}} b_{j}^{N_{k_{p}}} ds +   \int_{0}^{T} \sum_{j=l}^{\infty} J_{i,j} a_{i} b_{j} ds.\label{inequalities for annihilation integral}
	\end{align}
    We now analyze the second integral occurring on the right in \eqref{inequalities for annihilation integral}. Since $1 \leq i \leq N_{k_{p}}$, using the symmetry of the annihilation rates $J_{i,j}$, together with~\eqref{double sum first moment for B} for $\sigma = 0$, we obtain the following estimate:
		\begin{align*}
			\int_{0}^{T} \sum_{j=l}^{N_{k_{p}}} J_{i,j} a_{i}^{N_{k_{p}}}(s) b_{j}^{N_{k_{p}}}(s) ds \leq & \ \frac{1}{l}  \int_{0}^{T} \sum_{j=l}^{N_{k_{p}}} J_{i,j}  a_{i}^{N_{k_{p}}}(s)  j b_{j}^{N_{k_{p}}}(s) ds \\
			\leq & \frac{1}{l}   \int_{0}^{T} {a_{i}}^{N_{k_{p}}}(s) \sum_{j=l}^{N_{k_{p}}} jJ_{i,j} b_{j}^{N_{k_{p}}}(s) ds \\
		 	\leq & \frac{1}{l}  \int_{0}^{T} \sum_{j=1}^{N_{k_{p}}} {a_{j}}^{N_{k_{p}}}(s) \sum_{i=1}^{N_{k_{p}}} i J_{i,j} b_{i}^{N_{k_{p}}}(s) ds  \\
			\leq & \frac{1}{l}  \sum_{i=1}^{2 N_{k_{p}}} i b_{i}^{0} 
		 	\leq  \frac{1}{l} \|b^{0}\|_{1}.
		 \end{align*}
Applying the above estimate to~\eqref{inequalities for annihilation integral} gives
\begin{align}
	\int_{0}^{T} \left|\sum_{j=1}^{N_{k_{p}}} J_{i,j} a_{i}^{N_{k_{p}}} b_{j}^{N_{k_{p}}} 
	- \sum_{j=1}^{\infty} J_{i,j} a_{i} b_{j}\right| ds 
	&\leq \int_{0}^{T} \sum_{j=1}^{l-1} J_{i,j} 
	\left|a_{i}^{N_{k_{p}}} b_{j}^{N_{k_{p}}} - a_{i} b_{j}\right| ds 
	+ \frac{\|b^{0}\|_{1}}{l} \nonumber \\[6pt]
	&\quad + \int_{0}^{T} \sum_{j=l}^{\infty} J_{i,j} a_{i} b_{j} \, ds.
	\label{tempconvoliminequality}
\end{align}
Passing to the limit as $p \to \infty$ and observing that the leading term on the right of~\eqref{tempconvoliminequality} vanishes due to \eqref{zeroth moment bound}, \eqref{limit of relabeled seq a*}, \eqref{limit of relabeled seq b*}, \eqref{solfirstmomentestimates}, and DCT, we infer that
\begin{align}
	\lim_{p \to \infty} \int_{0}^{T} \left|\sum_{j=1}^{N_{k_{p}}} J_{i,j} a_{i}^{N_{k_{p}}} b_{j}^{N_{k_{p}}} 
	- \sum_{j=1}^{\infty} J_{i,j} a_{i} b_{j}\right| ds 
	&\leq \frac{\|b^{0}\|_{1}}{l} 
	+ \int_{0}^{T} \sum_{j=l}^{\infty} J_{i,j} a_{i} b_{j} \, ds.
	\label{tempconvoliminequality1}
\end{align}
Since $$\sum_{j=1}^{\infty} J_{i,j}\,a_{i} b_{j} \in L^{1}(0,T) $$\text{by} \eqref{L1L2integrability},
DCT yields,  
for each $T \in (0,\infty)$, \begin{align}\label{temptailconvoannihilation}
\lim_{l \to \infty} \int_{0}^{T} \sum_{j=l}^{\infty} J_{i,j} a_{i}(s) b_{j}(s) \, ds = 0.
\end{align}
Next, letting $l \to \infty$ in~\eqref{tempconvoliminequality1} and using~\eqref{temptailconvoannihilation}, we obtain~\eqref{convoannihilationA}. Similarly,~\eqref{convoannihilationB} holds, thereby completing the proof of~\Cref{mainconvolemmaforthmi}.
\end{proof}
With the preceding results, we can now prove \Cref{finite mass theorem} under the assumptions \eqref{eqn 2.3} and \eqref{eqn 2.4} on the coagulation rates.
\begin{proof}[Proof of \Cref{finite mass theorem}]
For $1 \leq i \leq N_{k_p}$ and $t\ge 0$, the truncated system~\eqref{truncated system} yields
\begin{align*}
    a_i^{N_{k_p}}(t) &= a_i^{0}+\int_{0}^{t} \left(\frac{1}{2} \sum_{j=1}^{i-1} K_{j,i-j}^{N_{k_p}}\,a_{j}^{N_{k_p}} a_{i-j}^{N_{k_p}}
			- \sum_{j=1}^{N_{k_p}} K_{i,j}^{N_{k_p}}\,a_{i}^{N_{k_p}} a_{j}^{N_{k_p}}
			- \sum_{j=1}^{N_{k_p}} J_{i,j}^{N_{k_p}}\,a_i^{N_{k_p}} b_{j}^{N_{k_p}}\right)\,ds,\\
    b_i^{N_{k_p}}(t) &= b_i^{0}+ \int_{0}^{t} \left(\frac{1}{2} \sum_{j=1}^{i-1} K_{j,i-j}^{N_{k_p}}\,b_{j}^{N_{k_p}} b_{i-j}^{N_{k_p}}
			- \sum_{j=1}^{N_{k_p}} K_{i,j}^{N_{k_p}}\,b_{i}^{N_{k_p}} b_{j}^{N_{k_p}}
			- \sum_{j=1}^{N_{k_p}} J_{i,j}^{N_{k_p}}\,b_i^{N_{k_p}} a_{j}^{N_{k_p}}\right)\,ds.
\end{align*}
Invoking \eqref{solfirstmomentestimates}, \eqref{limit of relabeled seq a*}, \eqref{limit of relabeled seq b*}, together with \eqref{tail of C series for A}--\eqref{convoannihilationB}, we let $p\to\infty$ in the equations above; it follows that $c=\left((a_i)_{i\ge 1},(b_i)_{i\ge 1}\right)$ satisfies \eqref{eqn 2.1}--\eqref{eqn 2.2} on $[0,\infty)$. Moreover, \eqref{eqn 2.1}--\eqref{eqn 2.2}, together  with \eqref{L1L2integrability}, imply that $a_{i},b_{i} \in C([0,\infty))$. Hence, \Cref{finite mass theorem} follows, provided that the coagulation rates satisfy \eqref{eqn 2.3} and \eqref{eqn 2.4}.
\end{proof}
 We next prove \Cref{finite mass theorem}, assuming \eqref{eqn 2.3} and \eqref{eqn 2.5} on the coagulation rates. Before doing so, we establish the following lemma, that provides bounds on $\bigg(\frac{da_{i}^{N}}{dt}\bigg)_{N\ge i}$and $\bigg(\frac{db_{i}^{N}}{dt}\bigg)_{N\ge i}$ in $L^{1}(0,T)$ for every $T>0$.
 \begin{lemma}\label{conc3 bounds**}
 Assume that \eqref{symmetry conditions}, \eqref{eqn 2.3}, and \eqref{eqn 2.5} hold, and that $a^{0},b^{0}\in X_{1}^{+}$. For all $i\geq1,$ $ N\geq i,$ and $T>0$, the following estimates hold:
      \begin{align}
	a_{i}^{N}(t) \leq \| a^{0} \|_{0} & \quad\text{and} \quad b_{i}^{N}(t)  \leq \| b^{0} \|_{0}, \ 	\text {for all} \  t \in [0,T] \label{conc4 bounds},\\
    &\left\| \frac{da_{i}^{N}}{dt} \right\|_{L^{1}(0,T)} \leq \gamma_{i}(T)\label{der bound for a}, \\
    & \left\|\frac{db_{i}^{N}}{dt} \right\|_{L^{1}(0,T)} \leq \delta_{i}(T)\label{der bound for b},\\
      \text{where}\;\;\;\gamma_{i}(T)&= \bigg( \frac{1}{2} \left(\sup_{1\leq j\leq i-1}K_{j,i-j}\right)\|a^{0}\|_{0}^{2}~T+ \sup_{j\geq1} \bigg(\frac{K_{i,j}}{j}\bigg)  \|a^{0}\|_{1}^{2}~T + \|a^{0}\|_{0}\bigg )\nonumber\\
    \text{and}\;\;\;\delta_{i}(T)&= \bigg(\frac{1}{2} \left(\sup_{1\leq j\leq i-1}K_{j,i-j}\right)\|b^{0}\|_{0}^{2}~T+ \sup_{j\geq1} \bigg(\frac{K_{i,j}}{j}\bigg) \|b^{0}\|_{1}^{2}~T + \|b^{0}\|_{0}\bigg).\nonumber
    \end{align}
     \end{lemma}
    \begin{proof}
 \eqref{conc4 bounds} follow immediately from \eqref{zeroth moment bound}. The proof of \eqref{der bound for a} and \eqref{der bound for b} follows from the idea inspired by \cite[Lemma~3]{leyvraz1981singularities}. Indeed, using \eqref{eqn 2.3}, \eqref{eqn 2.5}, \eqref{first moment bound}, and \eqref{double sum first moment for A} in the truncated system \eqref{truncated system}, we estimate as follows. 
For $1\le i\le N$, we have
\begin{align}\label{estimate for gamma_{i}}
	\bigg|\frac{da_{i}^{N}}{dt}\bigg|  \leq \bigg( \frac{1}{2} &\left(\sup_{1\leq j\leq i-1}{K_{j,i-j}}\right)\sum_{j=1}^{i-1}a_{j}^{N}a_{i-j}^{N}~T+\left(\sup_{j\geq1}\frac{K_{i,j}}{j}\right) \sum_{j=1}^{N} a_{i}^{N} j a_{j}^{N}\bigg)  +\sum_{j=1}^{N} J_{i,j} a_{i}^{N} b_{j}^{N}, \\
	\left\|\frac{da_{i}^{N}}{dt}\right\|_{L^{1}(0,T)} &\leq \bigg( \frac{1}{2} \left(\sup_{1\leq j\leq i-1}K_{j,i-j}\right)\|a^{0}\|_{0}^{2}~T+ \left(\sup_{j\geq1} \frac{K_{i,j}}{j}\right)  \|a^{0}\|_{1}^{2}~T + \|a^{0}\|_{0}\bigg).
\end{align}
By \eqref{eqn 2.3} and \eqref{eqn 2.5}, the term $\left(\sup_{j\geq1}\frac{K_{i,j}}{j}\right)$ appearing in \eqref{estimate for gamma_{i}} is finite.
Hence, \eqref{der bound for a} follows. Similarly, we can prove \eqref{der bound for b}.
\end{proof}
Using the estimates provided by the previous lemma, Helly's selection principle \cite[Section~36.5, pp.~372--374]{kolmogorov1975introductory} yields subsequences $ (a_{i}^{N_{k_{p}}})_{N_{k_{p}}\geq i}$ and  $(b_{i}^{N_{k_{p}}})_{N_{k_{p}}\geq i}$ of $(a_{i}^{N})_{N\ge i}$ and $(b_{i}^{N})_{N\ge i}$, respectively, together with sequences  $a=(a_{i})_{i\ge1}$ and  $b=(b_{i})_{i\ge1}$ such that 
\begin{align}\label{limit of relabeled seq**}
a_{i}^{N_{k_{p}}}(t) \rightarrow a_{i}(t),\quad 
b_{i}^{N_{k_{p}}}(t) \rightarrow b_{i}(t),\quad~~\text{as}~~p\to \infty,~~\text{for every}~i \geq 1 ~\text{and}~t\ge0.
\end{align}
 We now verify that $a(t), b(t) \in X_{1}^{+}$ for every $t \ge 0$. 
Indeed, using \eqref{first moment bound}, for $t\ge 0$ we obtain
\begin{align}\label{solmomentbound}
\sum_{j=1}^{\infty} j|g_{j,k}(t)| 
&= \lim_{M \to \infty} \lim_{N_{i} \to \infty} \sum_{j=1}^{M} j|g_{j,k}^{N_{i}}(t)| 
\le \lim_{M \to \infty} \sum_{j=1}^{N_{i}} j|g_{j,k}(0)| 
\le \sum_{j=1}^{\infty} j~|g_{j,k}^{0}|, \quad \text{for } k = 1,2,
\end{align}
where $g_{j,1} = a_{j}$, $g_{j,2} = b_{j}$, and $g^{0}_{j,1} = a^{0}_{j}$,  $g^{0}_{j,2} = b^{0}_{j}$ for all $j\ge 1$. The non-negativity of $a(t)$ and $b(t)$ follows from the fact that they are limits of the non-negative subsequences $\left(a_{i}^{N_{k_{p}}}(t)\right)_{{N_{k_{p}}}\ge i}$ and that $\left(b_{i}^{N_{k_{p}}}(t)\right)_{{N_{k_{p}}}\ge i}$, respectively, and the positive cone of the space $X_{1}$ is closed.
Finally,~since $a^{0},~ b^{0}\in X_{1}^{+},$ combining \eqref{eqn 2.3}, \eqref{eqn 2.5}, \eqref{first moment bound}, \eqref{double sum first moment for A}, \eqref{double sum first moment for B}, and \eqref{limit of relabeled seq**}, we deduce that
\begin{align}\label{L^{1}L^{2} integrability}
\sum_{j=1}^{\infty}K_{i,j}a_{j},\quad \sum_{j=1}^{\infty}K_{i,j}b_{j},\quad
\sum_{j=1}^{\infty} J_{i,j}\,a_{i} b_{j}, \quad 
\sum_{j=1}^{\infty} J_{i,j}\,b_{i} a_{j} \in L^{1}(0,T),
\end{align}
for every $T > 0$.

The following lemma establishes the convergence properties needed to pass to the limit.
\begin{lemma}
Suppose that the conditions \eqref{symmetry conditions}, \eqref{eqn 2.3}, and \eqref{eqn 2.5} hold, with $a^{0}, b^{0}\in X_{1}^{+}$. Then, for all $1\leq i \leq N_{k_p}$, and $T>0$, we have 
\begin{align}
	\lim_{p\to \infty} \left\|\sum_{j=1}^{N_{k_{p}}} K_{i,j}a_{j}^{N_{k_{p}}} - \sum_{j=1}^{\infty}K_{i,j}a_{j}\right\|_{L^{1}(0,T)}&=0,\label{conv of C series for A}\\
\lim_{p \to \infty} \left\|\sum_{j=1}^{N_{k_{p}}} K_{i,j} b_{j}^{N_{k_{p}}} - \sum_{j=1}^{\infty}K_{i,j}b_{j} \right\|_{L^{1}(0,T)} &=0,\label{conv of C series for B}\\
\lim_{p \to \infty}\bigg\|\sum_{j=1}^{N_{k_{p}}} J_{i,j} a_{i}^{N_{k_{p}}}(s) b_{j}^{N_{k_{p}}}(s) - \sum_{j=1}^{\infty} J_{i,j} a_{i}(s) b_{j}(s) \bigg\|_{L^1(0,T)}  &= 0, \label{ConvoannihilationA}\\
\text{and}\;\;
		\lim_{p \to \infty} \bigg\|\sum_{j=1}^{N_{k_{p}}} J_{i,j} b_{i}^{N_{k_{p}}}(s) a_{j}^{N_{k_{p}}}(s) -\sum_{j=1}^{\infty} J_{i,j} b_{i}(s) a_{j}(s) \bigg\|_{L^1(0,T)}  &= 0\label{ConvoannihilationB}. 
 	\end{align}	\begin{proof} The proofs of \eqref{conv of C series for A} and \eqref{conv of C series for B} follow from \eqref{eqn 2.3}, \eqref{eqn 2.5}, \eqref{zeroth moment bound}, \eqref{first moment bound}, \eqref{limit of relabeled seq**}, and DCT; see \cite[Section~2.2]{laurenccot2024redner} for a similar proof.
Similarly, \eqref{ConvoannihilationA} and \eqref{ConvoannihilationB} are obtained by applying the approach developed for \Cref{finite mass theorem}, assuming \eqref{eqn 2.3} and \eqref{eqn 2.4}.
\end{proof}
\end{lemma}
Finally, we establish \Cref{finite mass theorem} under the assumptions \eqref{eqn 2.3} and \eqref{eqn 2.5} on coagulation rates.
\begin{proof}[Proof of \Cref{finite mass theorem}]
For $1 \leq i \leq N_{k_p}$ and $t\ge 0$, the truncated system~\eqref{truncated system} yields
\begin{align*}
    a_i^{N_{k_p}}(t) &= a_i^{0}+\int_{0}^{t} \left(\frac{1}{2} \sum_{j=1}^{i-1} K_{j,i-j}^{N_{k_p}}\,a_{j}^{N_{k_p}} a_{i-j}^{N_{k_p}}
			- \sum_{j=1}^{N_{k_p}} K_{i,j}^{N_{k_p}}\,a_{i}^{N_{k_p}} a_{j}^{N_{k_p}}
			- \sum_{j=1}^{N_{k_p}} J_{i,j}^{N_{k_p}}\,a_i^{N_{k_p}} b_{j}^{N_{k_p}}\right)\,ds,\\
    b_i^{N_{k_p}}(t) &= b_i^{0}+ \int_{0}^{t} \left(\frac{1}{2} \sum_{j=1}^{i-1} K_{j,i-j}^{N_{k_p}}\,b_{j}^{N_{k_p}} b_{i-j}^{N_{k_p}}
			- \sum_{j=1}^{N_{k_p}} K_{i,j}^{N_{k_p}}\,b_{i}^{N_{k_p}} b_{j}^{N_{k_p}}
			- \sum_{j=1}^{N_{k_p}} J_{i,j}^{N_{k_p}}\,b_i^{N_{k_p}} a_{j}^{N_{k_p}}\right)\,ds.
\end{align*}
Using \eqref{first moment bound}, \eqref{limit of relabeled seq**}, \eqref{conv of C series for A}--\eqref{ConvoannihilationB}, and DCT, we let $p\to \infty$ in the equations above and deduce that $(c_{i})_{i\geq 1}=((a_{i})_{i\ge 1},(b_{i})_{i\geq 1})$  satisfies \eqref{eqn 2.1}--\eqref{eqn 2.2} 
on \( [0,\infty) \). Consequently, \eqref{eqn 2.1}--\eqref{eqn 2.2} together with~\eqref{solmomentbound} and \eqref{L^{1}L^{2} integrability} implies that $a_{i},b_{i}\in C([0,\infty))$. This establishes \Cref{finite mass theorem}.
\end{proof}
   We now turn to proving \Cref{finite particle theorem}, starting with the lemma below, which establishes that $\bigg(\frac{da_{i}^{N}}{dt}\bigg)_{N\ge i}$and $\bigg(\frac{db_{i}^{N}}{dt}\bigg)_{N\ge i}$ are bounded in $L^{2}(0,T)$ for every $T>0$.
    \begin{lemma}\label{Boundedness lemma}
Let the hypotheses of \Cref{finite particle theorem} hold. Then, for all $i\geq1$, $N\geq i$, and $T>0$, the following bounds are valid:
      \begin{align}
	a_{i}^{N}(t) \leq &\| a^{0} \|_{0} \quad\text{and} \quad b_{i}^{N}(t)  \leq \| b^{0} \|_{0}, \ 	\text {for all} \  t \in [0,T],\\
    \left\| \frac{da_{i}^{N}}{dt} \right\|_{L^{2}(0,T)} \leq & \Bigg\{ T^{1/2} \left( \sum_{j=1}^{i-1} K_{j,i-j} \right) \|a^{0}\|_{0}^2 + 2(1+\kappa)\rho_{i}\|a^{0}\|_{0}^{3/2}\Bigg\} +2\rho_{i}||a^{0}||_{0}||b^{0}||_{0}^{1/2},\\
  \text{and} ~~  \left\| \frac{db_{i}^{N}}{dt} \right\|_{L^{2}(0,T)} \leq &  \Bigg\{ T^{1/2} \left( \sum_{j=1}^{i-1} K_{j,i-j} \right) \|b^{0}\|_{0}^2 + 2 (1+\kappa)\rho_{i} \|b^{0}\|_{0}^{3/2}\Bigg\}+2\rho_{i}||b^{0}||_{0}||a^{0}||_{0}^{1/2}.
    \end{align}
       \end{lemma}
       \begin{proof}
The proof of \Cref{Boundedness lemma} is analogous to that of \cite[Theorem 2.2]{laurenccot1999global} and follows from \eqref{eqn 2.3}, \eqref{eqn 2.4}, and \eqref{zeroth moment bound}.      
        \end{proof}
It therefore follows from \Cref{Boundedness lemma} that $(a_{i}^{N})_{N\ge i}$ is uniformly bounded in $W^{1,2}(0,T)$, for every fixed $i\ge 1$ and $T\in(0,\infty)$. Owing to the compact embedding of $W^{1,2}(0,T)$ into $C([0,T])$(see\cite[Theorem 2.6.3]{Kesavan2019}), the sequence $(a_{i}^{N})_{N\ge i}$ is relatively compact in $C([0,T]).$ Consequently, an application of the diagonal argument yields a subsequence $(a_{i}^{N_{k_{p}}})_{N_{k_{p}\ge i}}$ of $(a_{i}^{N})_{N\ge i}$ together with the sequence $a=(a_{i})_{i\ge 1}$ of non-negative continuous functions defined on $[0,\infty)$ satisfying  
\begin{align}\label{limit of relabeled seq of a}
 \lim_{p \to \infty}\|a_{i}^{N_{k_{p}}}-a_{i}\|_{C([0,T])}=0,~ \text{for every} ~T\in(0,\infty).
\end{align}
A similar argument ensures the existence of a subsequence $(b_{i}^{N_{k_{p}}})_{N_{k_{p}}\ge i}$ of $(b_{i}^{N})_{N\ge i}$ together with the sequence $b=(b_{i})_{i\ge 1}$ of non-negative continuous functions defined on $[0,\infty)$ satisfying  
\begin{align}\label{limit of relabeled seq of b}
\lim_{p \to \infty}\|b_{i}^{N_{k_{p}}}-b_{i}\|_{C([0,T])}=0, ~\text{for every} ~T>0.
\end{align}
For $t\ge 0$, $l\geq 1$, and $N_{k_{p}} \geq l$, \eqref{zeroth moment bound} implies that
\begin{align} \label{truncated limit}
\sum_{i=1}^{l} a_{i}^{N_{k_{p}}}(t) \leq \|a^{0}\|_{0}, 
\qquad 
\sum_{i=1}^{l} b_{i}^{N_{k_{p}}}(t) \leq \|b^{0}\|_{0}. 
\end{align}
Applying \eqref{limit of relabeled seq of a} and \eqref{limit of relabeled seq of b}, and letting $p \to \infty$ in \eqref{truncated limit} yields
\[
\sum_{i=1}^{l} a_{i}(t) \leq \|a^{0}\|_{0}, 
\qquad 
\sum_{i=1}^{l} b_{i}(t) \leq \|b^{0}\|_{0}, 
\quad t \in [0, \infty).
\]
Since $l \geq 1$ is arbitrary, we obtain
\begin{align}\label{solzerothmomentestimates}
\sum_{i=1}^{\infty} a_{i}(t) \leq \|a^{0}\|_{0}, 
\qquad 
\sum_{i=1}^{\infty} b_{i}(t) \leq \|b^{0}\|_{0}, 
\quad t \in [0, \infty).
\end{align}  
Consequently,  
\[
c(t) = ((a_{i}(t))_{i\ge 1}, (b_{i}(t))_{i \geq 1}) \in X_{0}^{+} \times X_{0}^{+}, 
\quad t \in [0, \infty).
\]
Now, for any $M \geq 1$ and $t>0$, applying \eqref{S inequality} shows that,
	\begin{align*} \sum_{i=1}^{M} i a_{i}^{N_{k_{p}}}(t) \leq \frac{2}{S} \|a^{0}\|^{1/2}_{0} t^{-1/2}  \   \  \text{for} \  \ N_{k_{p}} \geq M.\end{align*}
	Using \eqref{limit of relabeled seq of a} and by letting $ p \to \infty$ in above inequality, there holds
	\begin{align*}
	 \sum_{i=1}^{M} i a_{i}(t) \leq \frac{2}{S} \|a^{0}\|^{1/2}_{0} t^{-1/2} \  \text{for} \  t > 0.
	 \end{align*}
	Since the above inequality is true for any $M\geq 1$, we get
\begin{align}\label{estimatefora(tau)}
\|a(t)\|_{1} \leq \frac{2}{S} \|a^{0}\|^{1/2}_{0} t^{-1/2} \  \text{for each } t >0.
\end{align}
The above estimate can similarly be obtained for $b(t)$, that is, we have 
\begin{align}\label{estimateforb(tau)}
	\|b(t)\|_{1} \leq \frac{2}{S} \|b^{0}\|^{1/2}_{0} t^{-1/2} \  \text{for each } t > 0.
\end{align}
An application of~\eqref{eqn 2.3}, \eqref{eqn 2.4}, \eqref{zeroth moment bound}, \eqref{limit of relabeled seq of a}, \eqref{limit of relabeled seq of b}, and Fatou's lemma yields the following:
\begin{align}\label{L2integrability}
\sum_{j=1}^{\infty} K_{i,j}\,a_{j},\; \sum_{j=1}^{\infty} K_{i,j}\,b_{j} \in L^{2}(0,T)
\end{align} 
for each fixed $i\ge 1$ and $T>0$. Further, since $a^{0},~b^{0} \in X_{0}^{+}$, using \eqref{double sum first moment for A}, \eqref{double sum first moment for B}, \eqref{limit of relabeled seq of a}, \eqref{limit of relabeled seq of b}, and Fatou's lemma, we deduce that
\begin{align}\label{L1integrability}
    \sum_{j=1}^{\infty} J_{i,j}\,a_{i} b_{j},\; \sum_{j=1}^{\infty} J_{i,j}\,b_{i} a_{j} \in L^{1}(0,T).
\end{align}
The following lemma establishes the convergence of the coagulation and annihilation terms, needed to take the limit in the integral formulation of the truncated system \eqref{truncated system}.
    \begin{lemma}\label{mainconvolemmaforthmi2}
Assuming the conditions of \Cref{finite particle theorem} are satisfied, then for each $1\leq i \leq N_{k_p}$,~$\tau > 0$, and $T \in (\tau,\infty)$, we have the following convergences:
\begin{align}\label{Tail of C series for A}
	\lim_{p \to \infty} \left\| \sum_{j=1}^{N_{k_p}} K_{i,j}\,a_{j}^{N_{k_p}} - \sum_{j=1}^{\infty} K_{i,j}\,a_{j} \right\|_{L^{2}(\tau,T)} &= 0, \\
	\lim_{p \to \infty} \left\| \sum_{j=1}^{N_{k_p}} K_{i,j}\,b_{j}^{N_{k_p}} - \sum_{j=1}^{\infty} K_{i,j}\,b_{j} \right\|_{L^{2}(\tau,T)} &= 0, \label{Tail of C series for B}\\
		\lim_{p \to \infty}\bigg\|\sum_{j=1}^{N_{k_{p}}} J_{i,j} a_{i}^{N_{k_{p}}}(s) b_{j}^{N_{k_{p}}}(s) - \sum_{j=1}^{\infty} J_{i,j} a_{i}(s) b_{j}(s) \bigg\|_{L^1(\tau,T)}  &= 0, \label{ConvoannihilationAtau}\\ \text{and}
		~\lim_{p \to \infty} \bigg\|\sum_{j=1}^{N_{k_{p}}} J_{i,j} b_{i}^{N_{k_{p}}}(s) a_{j}^{N_{k_{p}}}(s) -\sum_{j=1}^{\infty} J_{i,j} b_{i}(s) a_{j}(s) \bigg\|_{L^1(\tau,T)}  &= 0.  \label{ConvoannihilationBtau}
\end{align}
\end{lemma}
\begin{proof}
Let $\tau>0$ and $T\in(\tau,\infty).$ Since $a(\tau),~ b(\tau)\in X_{1}^{+}$, an argument analogous to that used in the proofs of \eqref{tail of C series for A} and \eqref{tail of C series for B}, yields, \eqref{Tail of C series for A} and \eqref{Tail of C series for B}.
To establish \eqref{ConvoannihilationAtau}, it suffices to show that 
\begin{align*}
\limsup_{p \to \infty}\bigg\|\sum_{j=1}^{N_{k_{p}}} J_{i,j} a_{i}^{N_{k_{p}}}(s) b_{j}^{N_{k_{p}}}(s) - \sum_{j=1}^{\infty} J_{i,j} a_{i}(s) b_{j}(s) \bigg\|_{L^1(\tau,T)} = 0. 
\end{align*}
Let $l$ be a fixed positive integer with $2 \le l < N_{k_p} < \infty$, where $T>0$ and $\tau \in (0,T)$. Applying the triangle inequality and using the non-negativity of the solutions $((a_{i}^{N_{k_{p}}})_{1\leq i \leq N_{k_{p}}}, (b_{i}^{N_{k_{p}}})_{1\leq i\leq N_{k_{p}}})$ and $((a_{i})_{i \ge 1}, (b_{i})_{i\ge 1})$, we infer that
		\begin{align}
			& \int_{\tau}^{T} \left|\sum_{j=1}^{N_{k_{p}}}  J_{i,j}  a_{i}^{N_{k_{p}}} b_{j}^{N_{k_{p}}} - \sum_{j=1}^{\infty} J_{i,j} a_{i} b_{j}\right| ds  \nonumber \\
            &\leq  \  \ \int_{\tau}^{T} \sum_{j=1}^{l-1} J_{i,j} \left|a_{i}^{N_{k_{p}}} b_{j}^{N_{k_{p}}} - a_{i} b_{j}\right| ds  +  \int_{\tau}^{T}		 \sum_{j=l}^{N_{k_{p}}} J_{i,j} a_{i}^{N_{k_{p}}} b_{j}^{N_{k_{p}}} ds +   \int_{\tau}^{T} \sum_{j=l}^{\infty} J_{i,j} a_{i} b_{j} ds.\label{annihilation inequality}
	\end{align}
   Next, we examine the second integral appearing on the right in \eqref{annihilation inequality}. Since $1 \leq i \leq N_{k_{p}}$, $\tau >0$, and by the symmetry of the annihilation rate $J_{i,j}$, together with~\eqref{double sum first moment for B}, \eqref{S inequality}, and \eqref{estimateforb(tau)}, we obtain the following estimates:
		\begin{align*}
			\int_{\tau}^{T} \sum_{j=l}^{N_{k_{p}}} J_{i,j} a_{i}^{N_{k_{p}}}(s) b_{j}^{N_{k_{p}}}(s) ds \leq \frac{1}{l}  \sum_{i=1}^{2 N_{k_{p}}} i b_{i}(\tau)  
		 	\leq  \frac{1}{l} \|b(\tau)\|_{1}\leq \frac{2}{l S}\|b^{0}\|^{1/2}_{0} \tau^{-1/2}.
		 \end{align*}
Applying the above estimate in~\eqref{annihilation inequality} gives
\begin{align}
	\int_{\tau}^{T} \left|\sum_{j=1}^{N_{k_{p}}} J_{i,j} a_{i}^{N_{k_{p}}} b_{j}^{N_{k_{p}}} 
	- \sum_{j=1}^{\infty} J_{i,j} a_{i} b_{j}\right| ds 
	&\leq \int_{\tau}^{T} \sum_{j=1}^{l-1} J_{i,j} 
	\left|a_{i}^{N_{k_{p}}} b_{j}^{N_{k_{p}}} - a_{i} b_{j}\right| ds 
	+ \frac{2}{l S}\|b^{0}\|^{1/2}_{0} \tau^{-1/2}\nonumber \\[6pt]
	&\quad + \int_{\tau}^{T} \sum_{j=l}^{\infty} J_{i,j} a_{i} b_{j} \, ds.
	\label{Tempconvoliminequality}
\end{align}
Taking the limit as $p \to \infty$, and observing that the leading term on the right 
of~\eqref{Tempconvoliminequality} vanishes due to \eqref{zeroth moment bound},
\eqref{limit of relabeled seq of a}, \eqref{limit of relabeled seq of b},  \eqref{solzerothmomentestimates}, and DCT, we deduce that
\begin{align}
	\lim_{p \to \infty} \int_{\tau}^{T} \left|\sum_{j=1}^{N_{k_{p}}} J_{i,j} a_{i}^{N_{k_{p}}} b_{j}^{N_{k_{p}}} 
	- \sum_{j=1}^{\infty} J_{i,j} a_{i} b_{j}\right| ds 
	&\leq \frac{2}{l S}\|b^{0}\|^{1/2}_{0} \tau^{-1/2}
	+ \int_{\tau}^{T} \sum_{j=l}^{\infty} J_{i,j} a_{i} b_{j} \, ds.
	\label{Tempconvoliminequality1}
\end{align}
From \eqref{L1integrability}, we have $$\sum_{j=1}^{\infty} J_{i,j}\,a_{i} b_{j} \in L^{1}(0,T).$$
Hence, by DCT,
\begin{align}\label{Temptailconvoannihilation}
\lim_{l \to \infty} \int_{\tau}^{T} \sum_{j=l}^{\infty} J_{i,j} a_{i}(s) b_{j}(s) \, ds = 0,
\end{align}
for every $T \in (\tau,\infty)$. Next, letting $l \to \infty$ in~\eqref{Tempconvoliminequality1} and using~\eqref{Temptailconvoannihilation}, we obtain~\eqref{ConvoannihilationAtau}. Similarly,~\eqref{ConvoannihilationBtau} holds. Thus,~\Cref{mainconvolemmaforthmi} is proved.
\end{proof}
We now turn to the proof of \Cref{finite particle theorem}.
 \begin{proof}[Proof of \Cref{finite particle theorem}]
        Let $i\geq1$ and $N_{k_{p}}\geq i$. From \eqref{truncated system}, we deduce that, for each $t>0$ and $\tau \in (0,t),$
            \begin{align}
a_{i}^{N_{k_{p}}}(t) &= a_{i}^{N_{k_{p}}}(\tau) 
   + \int_{\tau}^{t} \Bigg( \frac{1}{2} \sum_{j = 1}^{i-1} K_{j,i-j} 
      a_{j}^{N_{k_{p}}} a_{i-j}^{N_{k_{p}}} - a_{i} \sum_{j =1}^{\infty} K_{i,j} a_{j}^{N_{k_{p}}} -\sum_{j=1}^{\infty} J_{i,j} a_{i}^{N_{k_{p}}} b_{j}^{N_{k_{p}}}\Bigg) ds, \\
      b_{i}^{N_{k_{p}}}(t) &= b_{i}^{N_{k_{p}}}(\tau) 
   + \int_{\tau}^{t} \Bigg( \frac{1}{2} \sum_{j = 1}^{i-1} K_{j,i-j} 
      b_{j}^{N_{k_{p}}} b_{i-j}^{N_{k_{p}}} - b_{i} \sum_{j =1}^{\infty} K_{i,j} b_{j}^{N_{k_{p}}} -\sum_{j=1}^{\infty} J_{i,j} a_{j}^{N_{k_{p}}} b_{i}^{N_{k_{p}}} 
   \Bigg)ds .
\end{align}
Invoking \eqref{limit of relabeled seq of a}, \eqref{limit of relabeled seq of b},  \eqref{Tail of C series for A}, \eqref{Tail of C series for B}, \eqref{ConvoannihilationAtau}, and  \eqref{ConvoannihilationBtau}, and taking the limit as $p\to \infty$ we obtain
 \begin{align}
a_{i}(t) &= a_{i}(\tau) 
   + \int_{\tau}^{t} \Bigg( \frac{1}{2} \sum_{j = 1}^{i-1} K_{j,i-j} 
      a_{j}(s) a_{i-j}(s) - a_{i}(s) \sum_{j =1}^{\infty} K_{i,j} a_{j}(s) 
      - \sum_{j=1}^{\infty} J_{i,j} a_{i}(s) b_{j}(s) 
   \Bigg)\, ds ,\\
   b_{i}(t)& = b_{i}(\tau) 
   + \int_{\tau}^{t} \Bigg( \frac{1}{2} \sum_{j = 1}^{i-1} K_{j,i-j} 
      b_{j}(s) b_{i-j}(s) - b_{i}(s) \sum_{j =1}^{\infty} K_{i,j} b_{j}(s) 
      - \sum_{j=1}^{\infty} J_{i,j} b_{i}(s) a_{j}(s) 
   \Bigg)\, ds 
\end{align}
Since $a_i,b_i \in C([0,t])$ and the functions
\[
\sum_{j=1}^{\infty} K_{i,j}a_j,\quad
\sum_{j=1}^{\infty} K_{i,j}b_j,\quad
\sum_{j=1}^{\infty} J_{i,j}a_i b_j,\quad
\sum_{j=1}^{\infty} J_{i,j}b_i a_j
\]
belong to $L^{1}(0,t)$, we may take the limit $\tau \to 0$ in the preceding identities.
Consequently, $(c_{i})_{i\geq 1}=((a_{i})_{\geq 1},(b_{i})_{i\geq 1})$ satisfies \eqref{eqn 2.1}--\eqref{eqn 2.2} on $[0,\infty)$. This establishes the desired result of \Cref{finite particle theorem}.
	\end{proof}
\begin{proof}[Proof of \Cref{theorem 2.5}] \Cref{theorem 2.5} is proved similarly to \cite[Theorem~2.3]{laurenccot2002discrete}, and we thus omit it here. However, it is worth noting that, unlike in the coagulation-fragmentation model considered in \cite[Theorem~2.3]{laurenccot2002discrete}, where the mass density of solutions is conserved, the mass density under the assumptions of this theorem decreases over time.
\end{proof}
	\vspace{3 mm}
	\section{Uniqueness and continuous dependence}\label{Section Uniqueness}
	\noindent Solutions to \eqref{eqn 1.1}--\eqref{eqn 1.3} are not necessarily unique; see, for instance, \cite{escobedo2003gelation, tran2023local}. However, when additional assumptions are imposed on the coefficients, we can establish uniqueness. The uniqueness result follows directly from the continuous dependence result proved below. The proof follows the standard approach used in the analysis of coagulation models, involving the derivation of suitable estimates for the solutions and the subsequent application of Gr\"{o}nwall's inequality. Our approach is inspired by the idea in \cite[Theorem 4.1]{ball1990discrete} and \cite[Proposition 5.1]{MR3135638}.
	  
      \begin{prop}\label{prop 5.2}
Let $\alpha\in\left[0,\frac12\right]$, and suppose that
$
K_{i,j}\le \kappa_{2}(ij)^{\alpha},~
J_{i,j}\le \kappa_{3}(ij)^{\alpha},
$
where $\kappa_{2},\kappa_{3}>0$. Let
$C=(a^{1},b^{1})$ and $D=(a^{2},b^{2})$ be solutions of
\eqref{eqn 1.1}--\eqref{eqn 1.2}, according to \Cref{Definition}, with initial conditions
\[
C(0)=(a^{1}(0), b^{1}(0))=(a_0^1,b_0^1), \qquad
D(0)=(a^{2}(0), b^{2}(0))=(a_0^2,b_0^2),
\]
where
$a_0^1,a_0^2,b_0^1,b_0^2\in X_1^+$.
Then, for every $t\ge0$,
\[
\|(a^{1}(t),b^{1}(t))-(a^{2}(t),b^{2}(t))\|_{\beta}
\le
C_{5}(K',t)\,
\|(a^{1}(0),b^{1}(0))-(a^{2}(0),b^{2}(0))\|_{\beta},
\]
where $C_{5}(K',t)$ is a positive constant that depends on $K'$ and $t$, with
\[
K'=\max\{C_{1}+C_{4},\,C_{2}+C_{3}\},
\]
where
\[
\begin{aligned}
C_{1}=\kappa_{2}\bigl(\|a_0^1\|_{1}+\|a_0^2\|_{1}\bigr), ~
C_{2}=\kappa_{3}\|a_0^1\|_{1},~
C_{3}=\kappa_{2}\bigl(\|b_0^1\|_{1}+\|b_0^2\|_{1}\bigr),~
C_{4}=\kappa_{3}\|b_0^1\|_{1},
\end{aligned}
\]
and
\[
\|(x,y)\|_{\beta}=\|x\|_{\beta}+\|y\|_{\beta},
\]
where $\beta\in\mathbb{R}$ satisfies
$\alpha\le\beta$ and $\alpha+\beta\le1$, and
$x=(x_i)_{i\ge1}$, $y=(y_i)_{i\ge1}\in X_1^+$.
\end{prop}
	\begin{proof}
    
        Let
\[
C=(a^{1},b^{1})=\bigl((a^{1}_{j})_{j\ge1},(b^{1}_{j})_{j\ge1}\bigr)
\quad\text{and}\quad
D=(a^{2},b^{2})=\bigl((a^{2}_{j})_{j\ge1},(b^{2}_{j})_{j\ge1}\bigr)
\]
be two solutions of the system \eqref{eqn 1.1}--\eqref{eqn 1.3}, as specified in \Cref{Definition}. Define
\[
X=C-D=(a^{1}-a^{2},\,b^{1}-b^{2})
=(x^{1},x^{2})
=\bigl((x^{1}_{j})_{j\ge1},(x^{2}_{j})_{j\ge1}\bigr),
\]
where
\[
x^{1}_{j}=a^{1}_{j}-a^{2}_{j},
\qquad
x^{2}_{j}=b^{1}_{j}-b^{2}_{j},
\qquad j\ge1.
\]	Now, for $T>0$ and $t \in [0,T)$, we set
		\begin{align}\label{defined}
			v_{1}(t)= \sum _{i=1}^{\infty}  i^{\beta} | x^{1}_{i}(t)|,~ v_{2}(t) =  \sum _{i=1}^{\infty}  i^{\beta} | x^{2}_{i}(t)|,  \    
		\  	w^{1}_{i} = i^\beta \operatorname{sgn}(x^{1}_{i}), \ w^{2}_{i} = i^\beta \operatorname{sgn}(x^{2}_{i})
		\end{align}
for $i\ge 1$, where, for $x\in\mathbb{R}$, the sign function is given by
\[
\operatorname{sgn}(x)=
\begin{cases}
1, & x>0,\\
-1, & x<0,\\
0, & x=0.
\end{cases}
\]
\noindent Observe that if $\eta$ is an absolutely continuous function of $t$, then the function $t\mapsto |\eta(t)|$ is absolutely continuous as well, and consequently
\begin{align}\label{distributional derivative}
\frac{d}{dt}|\eta(t)| = sgn(\eta(t))\frac{d \eta}{dt}(t)~~\text{for almost every}~t>0.
\end{align}
Substituting $\eta(t) = x^{1}_{i}(t)$ and $\eta(t)=x^{2}_{i}(t)$ into \eqref{distributional derivative}, multiplying by $i^{\beta}$, and summing from $i=1$ to $n$, we get
\begin{align}\label{distributional derivative for a}
\frac{d}{dt}\sum_{i=1}^{n} i^{\beta}|x^{1}_{i}(t)|=\sum_{i=1}^{n}i^{\beta}sgn(x^{1}_{i}(t))\frac{d x^{1}_{i}}{dt}(t)
\end{align}
and
\begin{align}\label{distributional derivative for b}
\frac{d}{dt}\sum_{i=1}^{n} i^{\beta}|x^{2}_{i}(t)|=\sum_{i=1}^{n}i^{\beta}sgn(x^{2}_{i}(t))\frac{d x^{2}_{i}}{dt}(t),
\end{align}
for almost every $t>0$. Integrating both sides of \eqref{distributional derivative for a} over $[0,t]$ for any $t\in(0,\infty)$, and then applying
\cite[Lemma 3.1]{ball1990discrete} and following the approach used in the proof of
\cite[Theorem 4.1]{ball1990discrete}, we deduce
\begin{align}\label{estimate on derivatives of a and b}
        \sum_{i=1}^{n} i^{\beta} |x^{1}_{i}(t)|-\sum_{i=1}^{n} i^{\beta} |x^{1}_{i}(0)| & = \int_{0}^{t} [U_{n}(s)+W_{n}(s)] ds\nonumber\\
        &\quad- \int_{0}^{t}\sum_{i=1}^{n} i^{\beta} sgn(x_{i}^{1}) \sum_{j=1}^{\infty} J_{i,j} (a^{1}_{i} b^{1}_{j} - a^{2}_{i} b^{2}_{j})ds,
        \end{align}
     where
\begin{align*}
&V_{i,j}(u):=K_{i,j}u_i u_j,
\qquad
u\in\{a^1,a^2,b^1,b^2\},\\
M_{i,j}(a):=&V_{i,j}(a^1)-V_{i,j}(a^2),
\quad \text{and} \quad
M_{i,j}(b):=V_{i,j}(b^1)-V_{i,j}(b^2),
\end{align*}
for $i,j\geq 1$ and 
    	\begin{align*}
	&U_{n}(s) = \frac{1}{2} \sum_{i+j \leq n} (w^{1}_{i+j} - w^{1}_{i} - w^{1}_{j})\, M_{i,j}(a)= \frac{1}{2} \sum_{i+j \leq n} (w^{1}_{i+j} - w^{1}_{i} - w^{1}_{j})\, K_{i,j}(a^{1}_{i}x^{1}_{j}+a^{2}_{j}x^{1}_{i})\\ &\text{and}~~
    W_{n}(s)= -\sum_{i=1}^{n} \sum_{j = n - i + 1}^{\infty} w^{1}_{i}\, M_{i,j}(a)= -\sum_{i=1}^{n} \sum_{j = n - i + 1}^{\infty} w^{1}_{i}K_{i,j}(a^{1}_{i}x^{1}_{j}+a^{2}_{j}x^{1}_{i}),
\end{align*}
 for $s\in[0,t]$.
We begin by estimating the first term in $U_n(s)$ using the bound $K_{i,j}\le \kappa_{2}(ij)^\alpha$, the definition of the $\operatorname{sgn}$ function, and the condition $\alpha\le\beta$, as follows:
	\begin{align}\label{estimate for a^{1}_{i}}
	\frac{1}{2} \sum_{i+j\leq n} (w^{1}_{i+j}-w^{1}_{i}-w^{1}_{j})\, K_{i,j}\, a^{1}_{i}\, {x^{1}_{j}}
	&\;\leq\;
	\kappa_{2}\sum_{i+j\leq n} i^{\beta} i^{\alpha} j^{\alpha}\, a^{1}_{i}\, |x^{1}_{j}|\nonumber \\
	&= \kappa_{2} \sum_{i=1}^{\infty} i\, a^{1}_{i} 
	\sum_{j=1}^{\infty} j^{\alpha} |x^{1}_{j}|
   \;\leq \;\kappa_{2}\, \|a^{1}_{0}\|_{1}\, v_{1}.
	\end{align}
Here, we have used the inequality
$
\sum_{i=1}^{\infty}i a^{1}_{i}(t)
\leq
\sum_{i=1}^{\infty}i a^{1}_{i}(0),~ t\in(0,\infty),
$ which is a consequence of the argument establishing \Cref{finite mass theorem}. An analogous estimate can be derived for the second term in $U_{n}(s):$
    \begin{align}\label{estimate for a^{2}_{i}}
	\frac{1}{2} \sum_{i+j\leq n}(w^{1}_{i+j}-w^{1}_{i}-w^{1}_{j})\, K_{i,j}\, a^{2}_{j}\, x^{1}_{i}
	\;\leq\;\kappa_{2}\, \|a^{2}_{0}\|_{1}\, v_{1}.
    \end{align}
       Using \eqref{estimate for a^{1}_{i}} together with \eqref{estimate for a^{2}_{i}}, we obtain
    \begin{align}\label{U_{n} estimate} 
       \int_{0}^{t} U_{n}(s) ds\, \leq C_{1} \int_{0}^{t} v_{1}(s) ds,\quad t>0 
    \end{align}
    where $C_{1} =  \kappa_{2}(\|a^{1}_{0}\|_{1} + \| a^{2}_{0}\|_{1}).$
    	Next, we establish that 
	\begin{align}\label{W_{n} estimate}
	  \lim_{n \to \infty} \int_{0}^{t} W_{n}(s) ds = 0,\quad t>0.  
      \end{align}
      Since
\begin{align}
\sum_{i=1}^{\infty}\sum_{j=1}^{\infty}|w_i^1M_{i,j}|
&\le
\kappa_{2}\|a_0^1\|_1^2
+\kappa_{2}\|a_0^2\|_1^2
+2\kappa_{2}\|a_0^1\|_1\|a_0^2\|_1,
\end{align}
and
\begin{align}
\lim_{n\to\infty}
\sum_{i=1}^{n}\sum_{j=1}^{n-i}w_i^1M_{i,j}
=
\sum_{i=1}^{\infty}\sum_{j=1}^{\infty}w_i^1M_{i,j},
\end{align}
we conclude that
\[
\lim_{n\to\infty}W_n(s)
=
\lim_{n\to\infty}
\sum_{i=1}^{n}\sum_{j=n-i+1}^{\infty}w_i^1M_{i,j}(s)
=0
\]
for every $s\in[0,t].$
Hence, \eqref{W_{n} estimate} follows from the DCT. Next, using \eqref{U_{n} estimate} and \eqref{W_{n} estimate}, we obtain
		\begin{align}\label{estimate for sum of U and W}
			\int_{0}^{t} (U_{n}(s)+W_{n}(s))\,ds
			\leq C_{1} \int_{0}^{t} v_{1}(s)\,ds + \int_{0}^{t}W_{n}(s)\, ds,\;\; t>0.
		\end{align}
		
		Moreover, the final term in \eqref{estimate on derivatives of a and b} admits the representation
		\begin{align}
			\nonumber
			-  \sum_{i=1}^{n}i^\beta \text{sgn}&(x^{1}_{i})\sum_{j=1}^{\infty}  J_{i,j}~(a^{1}_i~b^{1}_{j} - a^{2}_i~b^{2}_{j} ) \\
			\nonumber
			=  & -  \sum_{i=1}^{n}i^\beta \text{sgn}(x^{1}_{i})\sum_{j=1}^{\infty} J_{i,j}~(a^{1}_i~b^{1}_{j}- a_{i}^{1} b^{2}_{j}  + a^{1}_{i} b^{2}_{j}- a^{2}_i~b^{2}_{j}) \\
			\nonumber
			= &-\sum_{i=1}^{n}\sum_{j=1}^{\infty} i^{\beta}  sgn(x^{1}_{i})  a^{1}_{i}  x^{2}_j  J_{i,j} -\sum_{i=1}^{n}\sum_{j=1}^{\infty} i^{\beta}  sgn(x^{1}_{i} )  b^{2}_{j}  x^{1}_{i}  J_{i,j} \\
			\leq &\sum_{i=1}^{n}\sum_{j=1}^{\infty} i^{\beta}  a^{1}_{i}  | x^{2}_j | J_{i,j} -\sum_{i=1}^{n}\sum_{j=1}^{\infty} i^{\beta} b^{2}_{j}  |x^{1}_i |  J_{i,j}.\label{eqn 5.11}
		\end{align}
		Since $J_{i,j}\leq \kappa_{3}(ij)^{\alpha}$ and $\alpha \leq \beta,$ we deduce the following inequality,
		\begin{align}
			\nonumber
			\sum_{i=1}^{n}\sum_{j=1}^{\infty} i^{\beta}  a^{1}_{i}  |x^{2}_j | J_{i,j} 
			\leq  \ \kappa_{3} &\sum_{i=1}^{n} \sum_{j=1}^{\infty} i^{\beta} \ |x^{2}_j|\  (i j)^{\alpha}   \ a^{1}_{i}
			\leq \kappa_{3} \sum_{i=1}^{\infty} i a^{1}_{i} \ \sum_{j=1}^{\infty} j^{\alpha} |x^{2}_{j}|\\
			\leq {\kappa_{3}} & \|a^{1}_{0}\|_{1} \sum_{j=1}^{\infty} j^{\beta} |x^{2}_{j}| = C_{2}\sum_{i=1}^{\infty} i^{\beta} |x^{2}_{i}|,\label{eqn 5.12}
		\end{align}
		where $C_{2} =\kappa_{3}\|a^{1}_{0}\|_{1}.$ Using \eqref{eqn 5.11} and \eqref{eqn 5.12}, we obtain the following bound
		\begin{align}\label{estimate for integral of annihilation term}
			\int_{0}^{t}\bigg[	-  \sum_{i=1}^{n}i^\beta \text{sgn}&(x^{1}_{i})\sum_{j=1}^{\infty}  J_{i,j}~(a^{1}_i~b^{1}_{j} - a^{2}_i~b^{2}_{j} ) \bigg] ds \leq C_{2}  \int_{0}^{t} v_{2}(s)ds,\;\; t>0.
		\end{align}
		Now, inserting~\eqref{estimate for sum of U and W} and \eqref{estimate for integral of annihilation term} into~\eqref{estimate on derivatives of a and b}, we conclude that
\begin{align}\label{final estimates}
\sum_{i=1}^{n} i^{\beta} |x^{1}_{i}(t)|-\sum_{i=1}^{n} i^{\beta} |x^{1}_{i}(0)|&= \int_{0}^{t} [U_{n}(s)+W_{n}(s)]\, ds\, - \int_{0}^{t}\sum_{i=1}^{n} i^{\beta} sgn(x_{i}^{1}) \sum_{j=1}^{\infty} J_{i,j} (a^{1}_{i} b^{1}_{j} - a^{2}_{i} b^{2}_{j})ds.\nonumber\\
&\leq  C_{1} \int_{0}^{t}v_{1}(s)\,ds + \int_{0}^{t}W_{n}(s)\,ds + C_{2}\int_{0}^{t}v_{2}(s)\,ds.
\end{align}
Letting $n\to\infty$ in \eqref{final estimates} and employing \eqref{defined} and \eqref{W_{n} estimate}, we have
        \begin{align}\label{estimate for v_{1}}
			v_{1}(t)\leq   v_{1}(0) + C_{1}  \int_{0}^{t} v_{1}(s) ds  + C_{2}  \int_{0}^{t} v_{2}(s) ds,\;\; t>0.
		\end{align}
	By a similar argument, we get
		\begin{align}\label{estimate for v_{2}}
			v_{2}(t) \leq  v_{2}(0) + C_{3} \int_{0}^{t} v_{2}(s) ds  + C_{4}  \int_{0}^{t} v_{1}(s) ds,\;\; t>0,
		\end{align} 
		where $C_{3}= \kappa_{2}(
        \|b^{1}_{0}\|_{1} + \| b^{2}_{0}\|_{1} )$  and $C_{4}= \kappa_{3} \|b^{1}_{0}\|_{1}.$  Adding \eqref{estimate for v_{1}} and \eqref{estimate for v_{2}}, we get
		$$v_{1}(t) + v_{2}(t) \leq  v_{1}(0)+v_{2}(0) +K^{'} \int_{0}^{t} (v_{1}(s) + v_{2}(s)) ds,\;\; t>0,$$
		where $K^{'} = \max\{C_{1}+C_{4}, C_{2}+C_{3}\}.$
	It follows from Gr\"onwall's inequality that
		\begin{align}
			\sum_{i=1}^{\infty}i^{\beta}|x^{1}_{i}(t)| + \sum_{i=1}^{\infty} i^{\beta} |x^{2}_{i}(t)| \leq C_{5}(K^{'},t)\bigg(\sum_{i=1}^{\infty}i^{\beta}|x^{1}_{i}(0)| + \sum_{i=1}^{\infty} i^{\beta} |x^{2}_{i}(0)|\bigg),\;\; t>0.
		\end{align}
The proof of~\Cref{prop 5.2} is therefore complete.
	\end{proof}
\begin{proof}[Proof of \Cref{Uniqueness theorem}]
Consider two solutions of the system \eqref{eqn 1.1}--\eqref{eqn 1.3}
satisfying the conditions stated in \Cref{Definition}, denoted by 
	\begin{align*}
		&C =  (a^{1}, b^{1})= [(a^{1}_{j})_{j\geq 1}, (b^{1}_{j})_{j\geq 1}], \  \   \text{and}  \ \   D = (a^{2}, b^{2}) = [(a^{2}_{j})_{j\geq 1}, (b^{2}_{j})_{j \geq 1}].
	\end{align*}
	with the same initial data, i.e. $C(0)=D(0)$. Since $C(0)=D(0)$, by the non-negativity of the norm and \Cref{prop 5.2}, we obtain
   $$ \|(a^{1}(t),b^{1}(t))-(a^{2}(t),b^{2}(t))\|_{\beta}=||a^{1}(t)-a^{2}(t)||_{\beta}+||b^{1}(t)-b^{2}(t)||_{\beta}=0,\;\; t\geq 0.$$
   Therefore,
   \begin{align*}
       \sum _{i=1}^{\infty}  i^{\beta} | a^{1}_{i}(t)-a^{2}_{i}(t)|=0\;\;\text{and}\;\;\sum _{i=1}^{\infty}  i^{\beta} | b^{1}_{i}(t)-b^{2}_{i}(t)|=0,\;\; t\geq 0.
   \end{align*}
 Hence, $a^{1}_{i}(t) = a^{2}_{i}(t)$ and $b^{1}_{i}(t) = b^{2}_{i}(t)$ for all $i\geq 1$ and $t\geq 0$, thus proving uniqueness.
	\end{proof}
\section{Differentiability}\label{Section Differentiability}
\noindent We now show that the solution defined in \Cref{Definition} for the coagulation-annihilation system \eqref{eqn 1.1}--\eqref{eqn 1.3} is in fact a classical solution for some specific cases. To this end, we first establish the uniform convergence of the sequence of solutions on compact subsets of [0,T], for $T>0$. Moreover, differentiability is established under additional assumptions on the coagulation-annihilation rates, namely,
\begin{align}\label{eqn 6.1}
		K_{i,j}, J_{i,j}\leq \rho_{i} \rho_{j},\;\; i,j \ge 1\quad  \text{and}  \quad \lim_{i\to \infty}\frac{\rho_{i}}{i} = 0.
			\end{align}
	Motivated by~\cite[Result~3, pp. 3403]{leyvraz1981singularities} and \cite[Proposition 6.1]{da2021modelling}, we arrive at the result stated below.
\begin{prop}\label{differentiability prop}
Let $(a, b)$ be a solution to \eqref{eqn 1.1}--\eqref{eqn 1.3} on $[0,T]$, according to \Cref{Definition}, where $T>0$. Assume \eqref{eqn 6.1} holds and the solution satisfies
\begin{align}\label{massdec0}
    \|a(t)\|_{1}\leq \|a^{0}\|_{1},\quad \|b(t)\|_{1}\leq \|b^{0}\|_{1},\;\;t\geq 0.
\end{align}
Then, the solution is continuously differentiable.
\end{prop}
\begin{proof}
To establish that $a_{i}$ is continuously differentiable on $(0,T)$ for every $i\in \mathbb{N}$ and $T>0$, it suffices to show that that the second and third integrands appearing on the right of \eqref{eqn 2.1} depend continuously on time. To this end, we aim to establish that, as $M \to \infty$, the sequences of functions
\[
\sum_{j=1}^{M} K_{i,j} a_{j}
\quad \text{and} \quad 
\sum_{j=1}^{M} J_{i,j} b_{j}, \qquad M \in \mathbb{N},
\]
converge uniformly on $[0,T]$ to 
\[
\sum_{j=1}^{\infty} K_{i,j} a_{j}
\quad \text{and} \quad 
\sum_{j=1}^{\infty} J_{i,j} b_{j},
\]
respectively. Exploiting~\eqref{massdec0}, we obtain
\[
\Big| \sum_{j=1}^{M} K_{i,j} a_{j}(t) - \sum_{j=1}^{\infty} K_{i,j} a_{j}(t) \Big|
\leq \rho_{i} \left(\sup_{k > M} \frac{\rho_{k}}{k}\right) \, \|a(t)\|_{1}\leq \rho_{i} \left(\sup_{k > M} \frac{\rho_{k}}{k}\right) \, \|a^{0}\|_{1},\quad t\in [0, T]
\]
and similarly,
\[
\Big| \sum_{j=1}^{M} J_{i,j} b_{j}(t)- \sum_{j=1}^{\infty} J_{i,j} b_{j}(t) \Big|
\leq \rho_{i} \left(\sup_{k > M} \frac{\rho_{k}}{k}\right) \, \|b(t)\|_{1} \leq \rho_{i} \left(\sup_{k > M} \frac{\rho_{k}}{k}\right) \, \|b^{0}\|_{1},\quad t \in [0, T].
\]
Using \eqref{eqn 6.1}, we conclude that
\[
\lim_{M \to \infty} \ \sup_{t \in [0,T]} 
\Big| \sum_{j=1}^{M} K_{i,j} a_{j}(t)- \sum_{j=1}^{\infty} K_{i,j} a_{j}(t) \Big| = 0
\]
and
\[
\lim_{M \to \infty} \ \sup_{t \in [0,T]} 
\Big| \sum_{j=1}^{M} J_{i,j} b_{j}(t)- \sum_{j=1}^{\infty} J_{i,j} b_{j}(t)\Big| = 0.
\]
Hence, $\sum_{j=1}^{M} K_{i,j} a_{j}$ converges uniformly on $[0,T]$ to $\sum_{j=1}^{\infty} K_{i,j} a_{j}$, and $\sum_{j=1}^{M} J_{i,j} b_{j}$ converges uniformly on $[0,T]$ to $\sum_{j=1}^{\infty} J_{i,j} b_{j}$.
 Moreover, by \Cref{finite mass theorem} when coagulation rates satisfy the assumptions \eqref{eqn 2.3} and \eqref{eqn 2.5}, both $a_{i}$ and $b_{i}$ are continuous on $[0, T]$ for all $i \in \mathbb{N}$. Since, the finite sums $\sum_{j=1}^{M} K_{i,j} a_{j}$ and $\sum_{j=1}^{M} J_{i,j} b_{j}$ are continuous, the uniform limit theorem implies that the infinite sums $ \sum_{j=1}^{\infty} K_{i,j} a_{j}$ and $\sum_{j=1}^{\infty} J_{i,j} b_{j}$ are also continuous on $[0,T].$ Consequently, the second and third integrands in the expressions following the
equals signs in \eqref{eqn 2.1}--\eqref{eqn 2.2} are continuous. This proves \Cref{differentiability prop}.
	\end{proof}
 \section{On the large-time behaviour of solutions}\label{ Large-Time behaviour of solutions}
\noindent We now examine how solutions to \eqref{eqn 1.1}--\eqref{eqn 1.2} evolve as $t \to \infty$, based on the method of~\cite[Proposition~5.1]{laurenccot1999global}, \cite[Theorem 4.3]{carr1994asymptotic}, \cite[Proposition 7.1]{ali2024discrete}, and \cite[Proposition 10.2.1]{Banasiak2019}. 
\begin{prop}
	Suppose the conditions stated in \Cref{finite particle theorem} hold. 
Let $(a,b)= ((a_{i})_{i \ge 1}, (b_{i})_{i\geq 1})$ denote the constructed solution to \eqref{eqn 1.1}--\eqref{eqn 1.3} provided by \Cref{finite particle theorem}. 
Then, we have
$$\lim_{t \to \infty} \|a(t)\|_{1} = 0\quad \text{and} \quad \lim_{t \to \infty} \|b(t)\|_{1} = 0.$$
\begin{proof} By the argument used in the proof of \Cref{finite particle theorem}, for every $t>0$,
\[
\|a(t)\|_{1} \leq \frac{2}{S} \|a^{0}\|^{1/2}_{0} \, t^{-1/2}, 
\qquad 
\|b(t)\|_{1} \leq \frac{2}{S} \|b^{0}\|_{0}^{1/2} \, t^{-1/2}.
\]
Consequently, letting $t\to \infty$, we deduce
$$\lim_{t \to \infty} \|a(t)\|_{1} = 0\quad \text{and} \quad \lim_{t \to \infty} \|b(t)\|_{1} = 0.$$
\end{proof}
\end{prop}
The following theorem shows that each component of the solution \eqref{eqn 1.1}--\eqref{eqn 1.3}, given in \Cref{Definition}, admits a limit as $t \to \infty$. Moreover, under a positivity assumption on the coagulation kernel, these limiting values are necessarily zero.
\begin{theorem}\label{large-time behaviour theorem}
    Let $T\in(0,\infty),$ and $(a^0, b^{0}) =((a^{0}_{i})_{i \ge 1}, (b^{0}_{i})_{i\geq 1}) \in X_{1}^{+}\times X_{1}^{+}.$ Suppose $(a,b)=((a_{i})_{i \ge 1},(b_{i})_{i\ge 1}) \in X_{1}^{+}\times X_{1}^{+}$ is a solution to \eqref{eqn 1.1}--\eqref{eqn 1.3}, given in \Cref{Definition}. Then there exist  sequences 
\begin{align*}
(\bar{a}, \bar{b})= ((\bar{a_{i}})_{i\ge 1}, (\bar{b_{i}})_{i\ge 1}) \in X_{1}^{+}\times X_{1}^{+}
\end{align*}
such that 
\begin{align}
\lim_{t \to \infty} a_{i}(t) = \bar{a_{i}}, \quad  \lim_{t \to \infty} b_{i}(t) = \bar{b_{i}} \quad\text{for all} \quad i \ge 1.
\end{align}
Moreover, assuming $K_{i,i}>0$ holds for all $i\ge1$, it follows that 
\begin{align}\label{limit is zero}
\bar{a_{i}}=0\quad \text{and}\quad \bar{b_{i}}=0 \quad \text{for all}\quad i\ge 1.\end{align}
\begin{proof}
Since $(a,b)=((a_{i})_{i\ge 1},(b_{i})_{i\ge 1})$ is a solution to \eqref{eqn 1.1}--\eqref{eqn 1.3}, according to \Cref{Definition}, for every $0\le p < \sigma,$ we have 
\begin{align*}
a_{i}(\sigma)-a_{i}(p)&= \int_{p}^{\sigma}\bigg(\frac{1}{2} \sum_{j=1}^{i-1} K_{j,i-j}\,a_{j}(s) a_{i-j}(s)
	- \sum_{j=1}^{\infty} K_{i,j}\, a_{i}(s) a_{j}(s)   
	- \sum_{j=1}^{\infty} J_{i,j}\,a_i(s) b_{j}(s)\bigg)\,ds, \quad i \in \mathbb{N}.
    \end{align*}
For $l\ge 1$, summing the above equation over $i =1$ to $i=l$ on both sides, we obtain 
\begin{align}\label{summation}
 \sum_{i=1}^{l}a_{i}(\sigma)-\sum_{i=1}^{l}a_{i}(p)= \int_{p}^{\sigma}\bigg(\frac{1}{2} \sum_{i=1}^{l}\sum_{j=1}^{i-1} K_{j,i-j}&\,a_{j}(s) a_{i-j}(s)
	- \sum_{i=1}^{l}\sum_{j=1}^{\infty} K_{i,j}\, a_{i}(s) a_{j}(s)\nonumber\\ &  - \sum_{i=1}^{l}{\sum_{j=1}^{\infty} J_{i,j}\,a_i(s) b_{j}(s)\bigg)\,ds}.
\end{align}
Upon interchanging the order of summation and performing a change of variables in the first integrand on the right side of \eqref{summation}, we get  
\begin{align}\label{change of order of summation}
\frac{1}{2} \sum_{i=1}^{l} \sum_{j=1}^{i-1}K_{j,i-j} a_{j} a_{i-j} &= \frac{1}{2}\sum_{j=1}^{l-1} \sum_{i=j+1}^{l}K_{j,i-j} a_{j} a_{i-j}= \frac{1}{2} \sum_{i=1}^{l-1} \sum_{j=1}^{l-i}K_{i,j}a_{i}a_{j}.
\end{align}
Using \eqref{change of order of summation} in \eqref{summation}, we deduce 
\begin{align}\label{derivative inequality}
 \sum_{i=1}^{l}a_{i}(\sigma)-\sum_{i=1}^{l}a_{i}(p)&=\int_{p}^{\sigma}\bigg( -\frac{1}{2}\sum_{i=1}^{l-1}\sum_{j=1}^{l-i}K_{i,j} a_{i}(s) a_{j}(s)\nonumber\\&-\sum_{i=1}^{l} \sum_{j=l-i+1}^{\infty} K_{i,j}a_{i}(s) a_{j}(s)-
\sum_{i=1}^{l} \sum_{j=1}^{\infty}J_{i,j}a_{i}(s)b_{j}(s)\bigg)\,ds.
\end{align}
Since the integrands on the right of the above equality take non-positive values, owing to the non-negativity of $a_{i}$, $b_{i}$, $K_{i,j},$ and $J_{i,j}$, we infer that  
\begin{align}\label{decreasing sum condition}
\sum_{i=1}^{l}a_{i}(p)\ge
\sum_{i=1}^{l}a_{i}(\sigma)~~\text{for}~ p<\sigma.\end{align}
Therefore, for every $l\ge 1,$ the function
$$g_{l}(t)=\sum_{i=1}^{l}a_{i}(t),\;\; t\geq 0$$ is non-negative and non-increasing in time. Consequently, there exists a constant $g_{l}^{*}$ such that 
\begin{align*}
g_{l}(t) \to g_{l}^{*} \quad\text{as}\quad t \to \infty
\end{align*}
and 
\begin{align*}
a_{l}(t) \to \bar{a}_{l}\quad\text{as\quad t} \to \infty, 
\end{align*}
where $\bar{a_{1}}=g_{1}^{*}$ and, for all $l\ge 2,$
$\bar{a_l}=g_{l}^{*}-g_{l-1}^{*}\ge0.$
Furthermore, since $a=(a_{i})_{i\ge1} \in X_{1}^{+}$  for $t \ge 0$, \eqref{zeroth moment bound} gives
\begin{align}
\sum_{i=1}^{l}i\bar{a_{i}}&=\lim_{t \to \infty} \sum_{i=1}^{l}ia_{i}(t)\le \lim_{t \to \infty} \sum_{i=1}^{l}ia_{i}(0) = \sum_{i=1}^{l}ia_{i}(0)\le \sum_{i=1}^{\infty}ia_{i}(0).
\end{align}
Taking $l \to \infty,$ we infer
$$\sum_{i=1}^{\infty} i\bar{a_{i}}< \infty.$$
An analogous argument yields the existence of a subsequence $\bar{b}=(\bar{b}_{i})_{i\ge 1} \in X_{1}^{+}$ such that
\begin{align}
\lim_{t \to \infty} b_{i}(t)=\bar{b_{i}}.
\end{align}
We now establish \eqref{limit is zero}, i.e., $\bar{a_{l}}=0$ and $\bar{b_{l}}=0$ for every $l\ge 1$.  Let $(\psi_{i})_{i\ge 1}$ be a real-valued sequence that decays sufficiently rapidly. Since $(a_{i})_{i\ge 1}$ satisfies \eqref{eqn 2.1}, integrating \eqref{eqn 2.1} over the interval $[t,t+\sigma],$ multiplying the resulting identity by $\psi_{i}$, and summing over $i=1,....,l$, we obtain
\begin{align}\label{phi multiplication}
\sum_{i=1}^{l}\psi_{i}a_{i}(t+\sigma)-\sum_{i=1}^{l}\psi_{i}a_{i}(t)=\int_{t}^{t+\sigma}\bigg(\frac{1}{2} \sum_{i=1}^{l} \sum_{j=1}^{i-1} \psi_{i} K_{j,i-j} a_{j}(s) a_{i-j}(s)&-\sum_{i=1}^{l} \sum_{j=1}^{\infty}\psi_{i} K_{i,j} a_{i}(s) a_{j}(s)\nonumber\\&-\sum_{i=1}^{l} \sum_{j=1}^{\infty} \psi_{i} J_{i,j} a_{i}(s) b_{j}(s)\bigg)\,ds.
\end{align}
By taking $l=1$ and $\psi_{i}=1$ in the above equation, and then integrating, we get 
\begin{align*}
a_{1}(t+\sigma) - a_1(t)
&=  \int_t^{t+\sigma} \left[ -a_1^2(s)K_{1,1}
- a_1(s)\sum_{j=2}^{\infty} K_{1,j}\,a_j(s)-\sum_{j=1}^{\infty} J_{1,j}a_{1}(s)b_{j}(s) \right] ds \\
&\le - \int_t^{t+\sigma} K_{1,1}\,a_1^2(s)\,ds
\end{align*}
for all $t>0$ and $\sigma>0$. Taking the limit as $t\to \infty$, we get
\begin{align*}
\lim_{t \to \infty} \int_{t}^{t+\sigma} K_{1,1}a_{1}^{2}(s)\,ds=0.
\end{align*}
Hence, $\bar{a}_{1}=0$. Otherwise, there exists $0<\mu_{1}<a_{1}$ such that $a_{1}(t)>\mu_{1}$ for all large $t$, implying 
\begin{align*}
\lim_{t \to \infty} \int_{t}^{t+\sigma} K_{1,1}a_{1}^{2}(s)\,ds\ge K_{1,1}\mu_{1}^{2}\sigma>0,
\end{align*}
  a contradiction. Next, for $l=2, ~\psi_{1}=0$ and $\psi_{2}=1$ in \eqref{phi multiplication}, we obtain 
 \begin{align*}
 a_{2}(t+\sigma)-a_{2}(t)&=\int_{t}^{t+\sigma}\bigg[\frac{1}{2}K_{1,1}a_{1}^{2}(s)-K_{2,2}a_{2}^{2}(s)-\sum_{j=3}^{\infty}K_{2,j}a_{2}(s)a_{j}(s)-\sum_{j=1}^{\infty}J_{2,j}a_{2}(s)b_{j}(s)\,\bigg]ds\\
 &\le \int_{t}^{t+\sigma}\bigg[\frac{1}{2}K_{1,1}a_{1}^{2}(s)-K_{2,2}a_{2}^{2}(s)\bigg]\,ds.
 \end{align*}
 Taking $t\to\infty$, we get
 \begin{align*}
 0\le\lim_{t\to \infty} \int_{t}^{t+\sigma}\bigg(\frac{1}{2}K_{1,1}a_{1}^{2}(s)-K_{2,2}a_{2}^{2}(s)\bigg)\,ds
 \end{align*}
 and since $a_{1}(t)\to 0$ as $t \to \infty$ and $(K_{i,i})_{i\ge 1}>0,$ this implies that
 \begin{align*}
     \lim_{t \to \infty} \int_{t}^{t+\sigma}K_{2,2}a_{2}^{2}(s)\,ds=0.
 \end{align*}
 This, in turn, implies that $\bar{a_{2}}=0$. Indeed, if this were not the case, then there would exist a constant $\mu_{2}>0$ for which $a_{2}(t)>\mu_{2}$ for all sufficiently large $t$, which leads to a contradiction, since
 \begin{align*}
     \lim_{t \to \infty} \int_{t}^{t+\sigma}K_{2,2}{a}^{2}_{2}(s)\,ds\ge K_{2,2}\mu_{2}^{2}\sigma>0.
 \end{align*}
 By induction, assume that $\bar{a_1} = \ldots \bar{a_{l-1}}=0$. We aim to prove that $\bar{a_{l}}=0$. To this end, note that 
 \begin{align}
 0&=\lim_{t \to \infty}(a_{l}(t+\sigma)-a_{l}(t))\nonumber\\
 &=\lim_{t \to \infty}\bigg[\int_{t}^{t+\sigma}\bigg(\frac{1}{2}\sum_{j=1}^{l-1}K_{j,l-j} a_{j}(s) a_{l-j}(s)-\sum_{j=1}^{\infty}K_{l,j} a_{l}(s)a_{j}(s)-\sum_{j=1}^{\infty}J_{l,j}a_{l}(s)b_{j}(s)\bigg)\,ds\bigg].
 \end{align}
 Invoking the induction hypothesis, the first term vanishes in the limit as $t \to \infty$. Furthermore, by isolating the contribution corresponding to $l$ in the second summation, we obtain
 \begin{align}
 0\leq - K_{l,l}~ \lim_{t \to \infty} \int_{t}^{t+\sigma} (a_{l}(t))^2\,ds.
 \end{align}
Hence, $\bar{a_{l}}=0$.
Moreover, by repeating the arguments developed above, one can further conclude that $\bar{b}_{l}=0$ for all $l\ge 1$. The proof of \Cref{large-time behaviour theorem} is thereby complete.
\end{proof}
\end{theorem}
Next, it is demonstrated that, under a suitable hypothesis on the coagulation kernel, the total particle number decreases to zero as $t \to \infty$.
\begin{prop}
Let $T \in (0,\infty)$ and $(a^{0},b^{0})=((a^{0}_{i})_{i\ge 1}, (b^{0}_{i})_{i\ge 1}) \in X_{1}^{+} \times X_{1}^{+}$. Let $(a,b)=((a_{i})_{i\ge 1}, (b_{i})_{i\ge 1}) \in X_{1}^{+}\times X_{1}^{+}$, a solution of \eqref{eqn 1.1}--\eqref{eqn 1.3}, as introduced in \Cref{Definition}. Assume there is a constant $\gamma>0$ for which
\begin{align}\label{positivity condition on coagulation kernel}
K_{i,j}\ge \xi, \quad \text{for all}\quad i,j\ge1.
\end{align}
Then 
\begin{align*}
\lim_{t \to \infty} \sum_{i=1}^{\infty}a_{i}(t)=0,\quad \text{and} \quad \lim_{t \to \infty} \sum_{i=1}^{\infty}b_{i}(t)=0.
\end{align*}
\begin{proof}
The proof of the above proposition follows similar lines to those of \cite[Proposition 7.1]{ali2024discrete} and \cite[Proposition 10.2.1]{Banasiak2019}. 
\end{proof}
 \end{prop}
 We next turn our attention to the case of pure annihilation.
\begin{prop}
Assume that \eqref{eqn 6.1} holds and $(a^{0}, b^{0})=((a_{i}^{0})_{i \ge1},(b_{i}^{0})_{i\ge1}) \in X_{0}^{+}\times X_{0}^{+}$. Further, suppose that $K_{i,j} = 0$ and $J_{i,j} \ge \delta$ for some $\delta > 0$ and for all $i, j \geq 1$. Let $(a,b)= ((a_{i})_{i\ge1}, (b_{i})_{i\geq 1}$) denote the constructed solution to \eqref{eqn 1.1}--\eqref{eqn 1.3} provided by \Cref{differentiability prop} and satisfy
\begin{align}\label{massdec}
    \|a(t)\|_{1}\leq \|a^{0}\|_{1},\quad \|b(t)\|_{1}\leq \|b^{0}\|_{1},\;\;t\geq 0.
\end{align}
The following claims hold true:
\begin{itemize}
\item[(i)] If $\|a^{0}\|_{0} = \|b^{0}\|_{0}$, then
\begin{align}
\lim_{t \to \infty}\|a(t)\|_{0} = 0
\quad \text{and} \quad
\lim_{t \to \infty} \|b(t)\|_{0} = 0.
\end{align}
\item[(ii)] If $\|a^{0}\|_{0} \neq \|b^{0}\|_{0}$, then the long time behaviour is determined by larger initial mass. More precisely,
 if $\|a^{0}\|_{0} > \|b^{0}\|_{0}$, then $\|b(t)\|_{0} \to 0$ and $\|a(t)\|_{0} \to \|a^{0}\|_{0}-\|b^{0}\|_{0}$  as $t \to \infty$. Conversely, if $\|a^{0}\|_{0} < \|b^{0}\|_{0}$, then $\|a(t)\|_{0} \to 0$ and $\|b(t)\|_{0} \to \|b^{0}\|_{0}-\|a^{0}\|_{0}$ as $t \to \infty$.
\end{itemize}
\end{prop}
\begin{proof}
Since \eqref{eqn 6.1} holds, the solution to \eqref{eqn 1.1}--\eqref{eqn 1.2} is differentiable by \Cref{differentiability prop}. Moreover, as  $K_{i,j} = 0$ for all $i,j \geq 1$, (1.1)--(1.2) yield
\begin{equation}\label{diferential inequality for a}
\frac{da_i(t)}{dt}  = - \sum_{j=1}^{\infty} J_{i,j}\, a_i(t)\, b_j(t),\quad t>0,
\end{equation}
\begin{equation}\label{differential inequality for b}
\frac{db_j(t)}{dt} = - \sum_{i=1}^{\infty} J_{i,j}\, a_i(t)\, b_j(t), \quad t>0.
\end{equation}
Summing the above equations over $i \ge 1$ and $j \ge 1$, respectively, we obtain
\begin{align}\label{equality of differential inequalities}
\frac{d}{dt}\sum_{i=1}^{\infty}a_i(t) = - \sum_{i=1}^{\infty} \sum_{j=1}^{\infty} J_{i,j}\, a_i(t)\, b_j(t),
\qquad
\frac{d}{dt} \sum_{j=1}^{\infty}b_{j}(t) = - \sum_{i=1}^{\infty} \sum_{j=1}^{\infty} J_{i,j}\, a_i(t)\, b_j(t),
\end{align}
for all $t>0$. In the above calculation, we interchange the derivative and the summation by using \eqref{massdec} together with the uniform convergence of the double series on the right side of \eqref{equality of differential inequalities}. Moreover, it follows from \eqref{equality of differential inequalities} and the non-negativity of $(a_i)_{i\ge1}$, $(b_i)_{i\ge1}$, and $J_{i,j}$ for $i,j\ge1$ that the functions $t\mapsto\|a(t)\|_{0}$ and $t\mapsto\|b(t)\|_{0}$ are non-increasing and bounded below by $0$. Therefore,
\[
\lim_{t\to\infty}\|a(t)\|_{0}
\quad\text{and}\quad
\lim_{t\to\infty}\|b(t)\|_{0}
\]
exist. Now, from~\eqref{equality of differential inequalities}, we infer that
\begin{align}\label{differenceconservation}
\|a(t)\|_{0} - \|b(t)\|_{0} = \|a^{0}\|_{0} - \|b^{0}\|_{0} =: C,
\end{align}
for each $t>0$. Next, we consider two cases.
\begin{itemize}
\item[(i)] Suppose that $\|a^{0}\|_{0} = \|b^{0}\|_{0}$. Then $C = 0$, and consequently $\|a(t)\|_{0} = \|b(t)\|_{0}$ for all $t \ge 0$. Using the lower bound $J_{i,j} \ge \delta > 0$ in~\eqref{equality of differential inequalities}, we obtain
\begin{align}
\frac{d}{dt}\|a(t)\|_{0} = - \sum_{i=1}^{\infty} \sum_{j=1}^{\infty} J_{i,j}\, a_i(t)\, b_j(t) \le -\delta \|a(t)\|_{0} \|b(t)\|_{0}, \quad t>0.
\end{align}
Since $\|a(t)\|_{0} = \|b(t)\|_{0}$, we deduce that
\begin{align}
 \frac{d}{dt}\|a(t)\|_{0}\le -\delta \|a(t)\|_{0}^2, \quad t>0.
\end{align}
\end{itemize}
\noindent Solving the differential inequality above, we obtain
\[
\|a(t)\|_{0} \le \frac{\|a^{0}\|_{0}}{1 + \delta \|a^{0}\|_{0}t},\quad t>0,
\]
which implies $\lim_{t \to \infty} \|a(t)\|_{0} = 0$. Since $\|a(t)\|_{0} = \|b(t)\|_{0}$, we also obtain $\lim_{t \to \infty} \|b(t)\|_{0}= 0$.

\noindent
(ii) Assume that $\|a^{0}\|_{0}>\|b^{0}\|_{0}$. Therefore, we have $\|a^{0}\|_{0} -\|b^{0}\|_{0} = C>0$. Then, by using~\eqref{differenceconservation}, we get $\|a(t)\|_{0} = \|b(t)\|_{0} + C$ for all $t \ge 0$. Substituting this into \eqref{equality of differential inequalities}, we obtain
\[
 \frac{d}{dt}\|b(t)\|_{0} \le -\delta \|b(t)\|_{0}\big(\|b(t)\|_{0}+ C\big), \quad t>0.
\]
Separating the variables and integrating over $(0,t)$ yields
\[
\frac{\|b(t)\|_{0}}{\|b(t)\|_{0} + C} \le \frac{\|b^{0}\|_{0}}{\|b^{0}\|_{0} + C} e^{-\delta C t}, \quad t>0.
\]
Consequently, since $C>0$, we infer that
\[
\frac{\|b(t)\|_{0}}{\|b(t)\|_{0}+C}\to0
\qquad\text{as }t\to\infty.
\]
It then follows from $C>0$ and \eqref{differenceconservation} that
\[
\|b(t)\|_{0}\to0
\qquad\text{and}\qquad
\|a(t)\|_{0}\to C
\quad\text{as }t\to\infty.
\]
An analogous argument applies when $\|a^{0}\|_{0}< \|b^{0}\|_{0}$. In this case, we obtain
\[
\|a(t)\|_{0} \to 0 \quad \text{and} \quad \|b(t)\|_{0}\to \|b^{0}\|_{0} - \|a^{0}\|_{0}\quad \text{as } t \to \infty.
\]
Thus, in both cases, the asserted large-time behavior follows, and the proof is complete.
\end{proof}
\section*{Acknowledgement}
\noindent The first author extends sincere thanks to the Ministry of Education, Government of India and the Indian Institute of Technology Roorkee for funding this work through a Ph.D. fellowship.
The second author expresses sincere gratitude for the financial support received from the DSI/NRF SARChI Chair in Mathematical Models and Methods in Biosciences and Bioengineering through the University of Pretoria's Externally Funded Postdoctoral Fellowship Programme (Cost Centre: N00317/82770). The authors also thank Prof. Philippe Lauren{\c{c}}ot for his valuable insights.
	
	\bibliographystyle{abbrv}
	\bibliography{Coag_Annihilation}
\end{document}